\documentclass[11pt,a4paper]{article}
\usepackage[cp1251]{inputenc}
\usepackage{amsmath,amsthm}
\usepackage{amsfonts,amstext,amssymb,verbatim,epsfig}

\usepackage{color}             
\usepackage{graphicx}
\usepackage[]{amsmath}
\usepackage{amssymb}
\usepackage{hyperref}
\usepackage{enumerate, amsfonts, amsthm, bbm}

\definecolor{MyDarkBlue}{rgb}{0,0.08,0.50}
\definecolor{BrickRed}{rgb}{0.65,0.08,0}

\def\R{{\mathbb{R}}}

\def\Z{{\mathbb{Z}}}

\newcommand{\sT}{{\scriptscriptstyle {\mathbb T}}}
\newcommand{\sZ}{{\scriptscriptstyle {\mathbb Z}}}
\newcommand{\Qn}[1]{Q_{#1}}
\newcommand{\torus}{{\mathbb T}^d_r}

\newcommand{\LC}[1]{{\rm LC}_{#1}}
\newcommand{\e}{{\rm e}}

\newcommand{\eqn}[1]{\begin{equation} #1 \end{equation}}
\newcommand{\eqan}[1]{\begin{align} #1 \end{align}}

\newcommand{\sss}{\scriptscriptstyle}
\newcommand{\conn}{\longleftrightarrow}
\newcommand{\expec}{{\mathbb E}}

\theoremstyle{usual}
\newtheorem{theorem}{Theorem}[section]

\newtheorem{lemma}{Lemma}[section]
\newtheorem{proposition}{Proposition}[section]

\newtheoremstyle{likedef}
  {}%
  {}%
  {}%
  {\parindent}%
  {\bfseries}%
  {.}%
  {.5em}%
  {}%

\theoremstyle{likedef}

\newtheorem{remark}{Remark}

\numberwithin{equation}{section}

\begin{document}

\title{Cycle structure of percolation on high-dimensional tori}

\author{
Remco van der Hofstad
\thanks{Department of Mathematics and
Computer Science, Eindhoven University of Technology, P.O.\ Box 513,
5600 MB Eindhoven, The Netherlands. E-mail: {\tt
rhofstad@win.tue.nl}}
\and
Art\"{e}m Sapozhnikov
\thanks{Max-Planck Institute for Mathematics in the Sciences, 
Inselstrasse 22, 04103 Leipzig, Germany. 
E-mail: {\tt artem.sapozhnikov@mis.mpg.de}}
}

\maketitle


\begin{abstract}
In the past years, many properties of the largest connected components of critical percolation
on the high-dimensional torus, such as their sizes and diameter, have
been established. The order of magnitude of these quantities equals the one
for percolation on the complete graph or Erd\H{o}s-R\'enyi random graph, raising
the question whether the scaling limits of the largest
connected components, as identified by Aldous (1997), are also equal.

In this paper, we investigate the \emph{cycle structure} of
the largest critical components
for high-dimensional percolation on the torus $\{-\lfloor r/2\rfloor, \ldots, \lceil r/2\rceil -1\}^d$.
While percolation clusters naturally have many \emph{short}
cycles, we show that the \emph{long} cycles, i.e.,
cycles that pass through the boundary of the cube of width $r/4$
centered around each of their vertices, have length
of order $r^{d/3}$, as on the critical Erd\H{o}s-R\'enyi random graph.
On the Erd\H{o}s-R\'enyi random graph, cycles play an essential role
in the scaling limit of the large critical clusters, as
identified by Addario-Berry, Broutin and Goldschmidt (2010).

Our proofs crucially rely on various new estimates of probabilities of 
the existence of open paths in critical Bernoulli percolation on $\Z^d$ 
with constraints on their lengths. We believe these estimates are interesting in their own right. 
\end{abstract}

\begin{abstract}
Plusieurs propri\'et\'es du comportement des grandes composantes connexes de la percolation critique sur le tore 
en dimensions grandes ont \'et\'e r\'ecemment \'etablies, telles la taille et le diam\'etre. 
L'ordre de grandeur de ces quantit\'es est \'egal \`a celle de la percolation sur le graphe complet ou sur le graphe al\'eatoire de Erd\H{o}s-R\'enyi. 
Ce r\'esultat sugg\`ere la question \`a savoir si les limites d'\'echelles des plus grandes composantes connexes, 
telles qu'identifi\'ees par Aldous (1997), sont aussi \'egales.

Dans ce travail, nous \'etudions la structure des cycles des plus grandes composantes connexes 
pour la percolation critique en dimension grande sur le tore $\{-\lfloor r/2\rfloor, \ldots, \lceil r/2\rceil -1\}^d$. 
Alors que les amas de percolation ont plusieurs cycles courts, nous montrons que les cycles longs, 
c'est-\`a-dire ceux qui passent \`a travers la fronti\`ere de chacun des cubes de largeur $r/4$ centr\'es aux sommets du cycle, 
ont une longueur de l'ordre $r^{d/3}$, comme dans le cas du graphe al\'eatoire critique d'Erd\H{o}s-R\'enyi. 
Sur ce dernier, les cycles jouent un r\^ole essentiel dans la limite d'\'echelle des grands amas critiques tels qu'identifi\'es par Addario-Berry, Broutin and Goldschmidt (2010).

Les preuves sont bas\'ees de mani\`ere cruciale sur de nouveaux estim\'es de la probabilit\'es d'existence de chemins ouverts dans la percolation critique 
de type Bernouilli sur $\Z^d$ avec contraintes sur leurs longueurs. 
Ces estim\'es sont potentiellement int\'eressants en soi.
\end{abstract}

\section{Introduction and results}
\label{sec-intro}
In the past years, the investigation of percolation on various
high-dimensional tori has attracted tremendous attention.
In \cite{BCHSS05a, BCHSS05b}, the phase transition of the
largest connected component was investigated for percolation
on general high-dimensional tori, including the complete graph,
the hypercube in high dimensions, as well as finite-range
percolation in sufficiently high dimensions. The phase transition
of percolation on high-dimensional tori is {\it mean-field},
i.e., it shares many features with that on
the complete graph as identified in \cite{ErdRen60} (see, e.g.,
\cite{Aldo97,Boll01,JanKnuLucPit93,JanLucRuc00,LucPitWie94}).

In \cite{BCHSS05a}, the subcritical and critical behavior was investigated under
the so-called \emph{triangle condition}, a general assumption
on the underlying graph that ensures that the model is mean-field.
The critical behavior of the model was identified in terms of the
blow-up of the expected cluster size,
which identifies a window of critical values of the edge occupation
probabilities. For any parameter value in this critical window,
the largest connected component was shown to be of order
$V^{2/3}$, as on the complete graph, where $V$ denotes the number
of vertices in the graph. In \cite{BCHSS05b}, the triangle condition
was proved to hold for the above-mentioned examples.

The situation of finite-range high-dimensional tori, which in the graph
sense converge to the hyper-cubic lattice,  was brought
substantially further in \cite{HH, HH1}, where, among others, it was shown that
the percolation critical value on the infinite lattice lies inside
the scaling window. We now know that the largest connected components are all of order $V^{2/3}$,
that the maximal connected component $|{\mathcal C}_{\rm max}|$ satisfies
that $|{\mathcal C}_{\rm max}|V^{-2/3}$ and $V^{2/3}/|{\mathcal C}_{\rm max}|$
are tight sequences of random variables that are non-concentrated, and that
the diameter of large clusters is of order $V^{1/3}$. These results
(and more) are also known to hold on the Erd\H{o}s-R\'enyi random graph,
see e.g., \cite{Aldo97, NacPer07b}, as well as the monographs
\cite{Boll01, JanLucRuc00}. This raises the question whether
the scaling limits agree. We shall expand on this question
in Section \ref{sec-dis} below.

\subsection{Percolation in high dimensions}
\label{secModel}
We consider bond percolation on a graph $G$. For a given parameter $p\in[0,1]$,
this is the probability measure ${\mathbb P}_p$ on subgraphs of $G$ defined as follows.
We delete edges of $G$ with probability $(1-p)$ and otherwise keep them,
independently for different edges. The edges of the resulting random subgraph of
$G$ are called {\it open} and the deleted edges are called {\it closed}.
Connected components of this random subgraph are called {\it open clusters}.
The graphs we investigate in this paper are
(a) the $d$-dimensional torus ${\mathbb T}_r^d = \{-\lfloor r/2\rfloor,\ldots,\lceil r/2\rceil-1\}^d$;
and (b) the hypercubic lattice ${\mathbb Z}^d$,
where the dimension $d$ is supposed to be sufficiently large.
How large we need to take $d$ depends on the edge structure of $G$.
We consider two different settings:
(a) In the {\it nearest-neighbor model}, two vertices are connected
by an edge if they are nearest-neighbors on $G$.
With our choice of $G$, every vertex has $2d$ nearest-neighbors.
In this setting we take the dimension $d$ large enough.
(b) In the {\it spread-out model} with a parameter $L$, two vertices
are connected by an edge if there is a hypercube of size $L$ in $G$
that contains these vertices.
With our choice of $G$, every vertex has $(2L+1)^d-1$ neighbors.
Of course, we are only interested in the case when the size of
the torus is much larger than $L$. In the spread-out setting with
large enough $L$, we take the dimension $d>6$.

To justify our choice of dimension, we recall a number of well-known
results about percolation on $\mathbb Z^d$. For bond percolation on
${\mathbb Z}^d$ with $d>1$, there exists a critical probability $p_c\in(0,1)$
such that, for $p<p_c$, all open clusters are almost surely finite and,
for $p>p_c$, there is almost surely an infinite open cluster.
At $p=p_c$, it is widely believed that there is almost surely no infinite open cluster.
This fact has been shown for $d=2$ by Kesten \cite{Kesten1/2} and for sufficiently
large $d$ by Hara and Slade \cite{HS}.
Here, by sufficiently large $d$, we mean $d > 18$ for the nearest-neighbor model and
$d>6$ for the spread-out model with sufficiently large $L$.
Showing this for all $d>1$ remains a challenging open problem.

The main assumption that we use in the paper concerns an estimate on the probability that,
at criticality, two vertices $x$ and $y$ are in the same open cluster of bond percolation on $\mathbb Z^d$,
which we denote by $x\leftrightarrow y$. We assume that there exist
constants $D_1$ and $D_2$ such that, for all $x$ and $y$ in $\mathbb Z^d$,
    \begin{equation}
    \label{prop2pf}
    D_1(1+|x-y|)^{2-d}\leq {\mathbb P}_{p_c}(x\leftrightarrow y) \leq D_2(1+|x-y|)^{2-d}.
    \end{equation}
These bounds have been established using so-called \emph{lace-expansion} techniques,
for the nearest-neighbor model with large enough $d$ by Hara \cite{Hara2008}, and
for the spread-out model with $d>6$ by Hara, van der Hofstad and Slade \cite{HHS}.
In fact, these papers give asymptotic formulas for such probabilities, but for our purposes,
the bounds (\ref{prop2pf}) suffice.

It is believed that the estimates (\ref{prop2pf}) hold for the nearest-neighbor model with $d>6$,
however the proof of this fact is beyond the current methods.
It has been proved by Chayes and Chayes \cite{ChCh} (assuming the existence of critical exponents)
that the bounds (\ref{prop2pf}) are violated for $d<6$.
The dimension $d=d_c=6$ is usually referred to as the {\it upper critical dimension}.

A simple computation using the upper bound in (\ref{prop2pf}) shows that
    \begin{equation}
    \label{triangle-cond}
    \nabla(p_c)=\sum_{x,y} {\mathbb P}_{p_c}(0\leftrightarrow x)
    {\mathbb P}_{p_c}(x\leftrightarrow y)
    {\mathbb P}_{p_c}(y\leftrightarrow 0)<\infty.
    \end{equation}
The bound in \eqref{triangle-cond} is called the \emph{triangle condition},
and is believed to be true for $d>6$. The triangle condition implies that
the sub- and critical phases of percolation on ${\mathbb Z}^d$ behave
similarly to the ones on a tree, for example, many critical exponents
on ${\mathbb Z}^d$ are equal to those on the tree. Intuitively,
the geometry of large critical clusters trivializes, since the space
is so vast that far away clusters are close to being independent.
In recent years, a related condition has been proved to hold on the torus,
which implies that the critical behavior of large connected components
on the high-dimensional torus is similar to that on the complete graph.
Sometimes this is called \emph{random graph asymptotics} for percolation
on the high-dimensional torus.

In this paper, we study the \emph{cycle structure} of bond percolation on the $d$-dimensional torus
in the above two settings.
Despite the fact that, for any $p\in(0,1)$, the vertices in open cycles occupy a positive fraction of the torus,
which is not the case for the critical and subcritical Erd\H{o}s-R\'enyi random graph
(see \cite{LucPitWie94}),
most of such vertices belong only to \emph{short} cycles, such as open squares of four bonds.
Short cycles vanish in the scaling limit of large critical clusters,
and are thus irrelevant to the scaling limit.
Therefore, we focus on the existence of open long cycles,
where we say that a cycle is long when it passes through the boundary of the cube of width $r/4$
centered around each of its vertices.
Special cases of long cycles are \emph{non-contractible cycles},
which are cycles that cannot, when considered as continuous curves, be contracted
to a point, and thus wind around the torus at least once.

Our main results show that the mean number of vertices in open
long cycles grows like $V^{1/3}$, and that such cycles
(when they exist) contain order of $V^{1/3}$ vertices. Moreover, we
show that the probability of the existence of at least one
open long cycle in a large cluster is bounded away from 0 and 1,
uniformly in the volume of the graph. As we discuss
in more details below, this situation is analogous to the situation
on the complete graph, as investigated in \cite{AddBroGol10,Aldo97,LucPitWie94}.
We also refer the reader to \cite[pages 722-723]{LucPitWie94}
for the discussion of more refined results about the structure of connected components
of the critical Erd\H{o}s-R\'enyi random graph.

For simplicity of presentation, {\it we restrict ourselves hereafter to
the nearest-neighbor model}. The results of this paper still hold for the
spread-out model on the $d$-dimensional torus with $d>6$.

The remainder of this section is organized as follows. In Section
\ref{sec-res}, we describe our main results, in Section \ref{sec-extra},
we describe some results on critical percolation on $\Z^d$ and the torus
that are used in the proofs of our main results and are interesting in
their own right, and in Section \ref{sec-dis},
we discuss the results and their relation to the work on the
Erd\H{o}s-R\'enyi random graph.

\subsection{Main results}
\label{sec-res}
We start by introducing some notation. For $a\in\R$, we write $|a|$ for the absolute value of $a$, and,
for a site $x = (x_1,\ldots,x_d)\in\Z^d$, we write $|x|$ for $\max(|x_1|,\ldots,|x_d|)$, and 
$|x|_1$ for $\sum_{i=1}^d|x_i|$.
For $s>0$ and $x\in\Z^d$,
let $\Qn{s}(x) = \{y\in\Z^d\colon|y-x|\leq s\}$ and $\partial \Qn{s}(x) = \{y\in\Z^d\colon|y-x|=\lfloor s\rfloor \}$.
We write $\Qn{s}$ for $\Qn{s}(0)$ and $\partial \Qn{s}$ for $\partial \Qn{s}(0)$.

\medskip

For a positive integer $r$, we consider the torus $(\torus,\mathbb E_r^d)$ with 
${\mathbb T}_r^d =\{-\lfloor r/2\rfloor,\ldots,\lceil r/2\rceil -1\}^d$
and the edge set 
${\mathbb E}_r^d = \big\{\{x,y\} \in\torus\times\torus\colon\sum_{i=1}^d|(x_i-y_i)(\mbox{mod } r)|=1\big\}$. 
We often abuse notation and write $\torus$ for $(\torus,{\mathbb E}^d_r)$.
The vertex $0 = (0,\ldots,0)$ is called the origin. 
We denote the number of vertices in the torus or \emph{volume} by $V = r^d$.
For $p\in[0,1]$, we consider the probability space
$(\Omega_{\sT,p},{\mathcal F}_{\sT,p},{\mathbb P}_{\sT,p})$, where
$\Omega_{\sT,p} = \{0,1\}^{{\mathbb E}_r^d}$, ${\mathcal F}_{\sT,p}$
is the $\sigma$-field generated by the 
cylinders of $\Omega_{\sT,p}$, and ${\mathbb P}_{\sT,p}$ is a
product measure on $(\Omega_{\sT,p},{\mathcal F}_{\sT,p})$,
${\mathbb P}_{\sT,p} = \prod_{e\in{\mathbb E}_r^d}\mu_e$, where
$\mu_e$ is given by $\mu_e(\omega_e = 1) = 1-\mu_e(\omega_e = 0) = p$, for
vectors $(\omega_e)_{e\in{\mathbb E}_r^d}\in\Omega_{\sT,p}$.
We write ${\mathbb E}_{\sT,p}$ for the expectation with respect to
${\mathbb P}_{\sT,p}$.

We further consider the hypercubic lattice $(\Z^d,{\mathbb E}^d)$, where
the edge set is given by
${\mathbb E}^d = \big\{\{x,y\} \in\Z^d\times\Z^d\colon|x-y|_1=1\big\}$.
Again, we often abuse notation and write $\Z^d$ for $(\Z^d,{\mathbb E}^d)$.
For $p\in[0,1]$, we consider a probability space
$(\Omega_{\sZ,p},{\mathcal F}_{\sZ,p},{\mathbb P}_{\sZ,p})$, where
$\Omega_{\sZ,p} = \{0,1\}^{{\mathbb E}^d}$, ${\mathcal F}_{\sZ,p}$
is the $\sigma$-field generated by the finite-dimensional
cylinders of $\Omega_{\sZ,p}$, and ${\mathbb P}_{\sZ,p}$
is the product measure on $(\Omega_{\sZ,p},{\mathcal F}_{\sZ,p})$,
${\mathbb P}_{\sZ,p} = \prod_{e\in{\mathbb E}^d}\mu_e$, where $\mu_e$
is given by $\mu_e(\omega_e = 1) = 1 - \mu_e(\omega_e = 0) = p$, for
vectors $(\omega_e)_{e\in{\mathbb E}^d}\in\Omega_{\sZ,p}$.

In both settings, we say that an edge $e$ is {\it open} or {\it occupied} if
$\omega_e = 1$, and $e$ is {\it closed} or {\it vacant} if $\omega_e = 0$.

\medskip

The event that two sets of sites ${\mathcal K}_1,{\mathcal K}_2\subset\mathbb T_r^d$
are connected by an open path is denoted by
$\{{\mathcal K}_1 \leftrightarrow {\mathcal K}_2~\mbox{in}~\mathbb T_r^d\}$, 
and the event that ${\mathcal K}_1$ and ${\mathcal K}_2$ are connected by an open path 
of length (number of edges) at most $k$ is denoted by 
$\{{\mathcal K}_1 \stackrel{\leq k}\longleftrightarrow {\mathcal K}_2~\mbox{in}~\mathbb T_r^d\}$.
We write ${\mathcal C}_{\sT}(x)$ for the set of $y\in\mathbb T_r^d$ such that
$x\leftrightarrow y$ in $\mathbb T_r^d$.

\medskip

Similarly, the event that two sets of sites ${\mathcal K}_1,{\mathcal K}_2\subset\Z^d$
are connected by an open path is denoted by $\{{\mathcal K}_1 \leftrightarrow {\mathcal K}_2~\mbox{in}~\Z^d\}$, and the event that
${\mathcal K}_1$ and ${\mathcal K}_2$ are connected by an open path of length at most $k$ is denoted by 
$\{{\mathcal K}_1 \stackrel{\leq k}\longleftrightarrow {\mathcal K}_2\mbox{ in }\Z^d\}$.
We write ${\mathcal C}_{\sZ}(x)$ for the set of $y\in\mathbb Z^d$ such
that $x\leftrightarrow y$ in $\mathbb Z^d$.

\medskip

Finally, for $x,y\in\torus$, we write 
$\tau_{\sT,p}(x,y) = {\mathbb P}_{\sT,p}(x\leftrightarrow y\mbox{ in }\torus)$, and $\tau_{\sT,p}(x)=\tau_{\sT,p}(0,x)$, while,
for $x,y\in\Z^d$, $\tau_{\sZ,p}(x,y) =  {\mathbb P}_{\sZ,p}(x\leftrightarrow y\mbox{ in }\Z^d)$
and $\tau_{\sZ,p}(x)=\tau_{\sZ,p}(0,x)$.

\medskip

We call a nearest-neighbor path $\pi = (x(1),\ldots,x(m))$ in $\torus$ a {\it cycle} if it is edge-disjoint and $x(1) = x(m)$, 
i.e., $\pi$ is an (edge-)self-avoiding polygon.
We say that a cycle $\pi$ is {\it long} if for each $1\leq n\leq m$, the cycle $\pi$ has a vertex in  
$\partial \Qn{r/4}(x(n))$. Finally, we denote by $\LC{k}$ the event that the origin is in an open long
cycle of length at most $k$.

\medskip

For two functions $g$ and $h$ from a set ${\mathcal X}$ to ${\mathbb R}$,
we write $g(z) \asymp h(z)$ to indicate that $g(z)/h(z)$ is bounded away from $0$ and $\infty$, uniformly in $z\in{\mathcal X}$.
All the constants $(C_i)$ in the proofs are strictly positive and finite
and depend only on the dimension, unless the dependence on other parameters is explicitly stated.
Their exact values may be different from section to section.

\medskip

We first give bounds on the probability that a vertex of the torus is in
an open long cycle.

\begin{theorem}[Expected number of vertices in long cycles]
\label{thmNCCExponent}
Assume (\ref{prop2pf}). 
For $x\in {\mathbb T}_r^d$,
    \begin{equation}
    {\mathbb P}_{\sT,p_c}(x~\mbox{is in an open long cycle})\asymp V^{-2/3}.
    \end{equation}
Consequently,
    \begin{equation}
    {\mathbb E}_{\sT,p_c}\big[\#\{x\colon x~\mbox{is in an open long cycle}\}\big]
    \asymp V^{1/3}.
    \end{equation}
\end{theorem}

\medskip

In the next theorem we show that, with high probability, large open clusters of
the torus may only contain long cycles of length of order $V^{1/3}$. 

\begin{theorem}[Long cycles have length of order $V^{1/3}$]
\label{thmNCCExistence}
Assume (\ref{prop2pf}). 
There exists $C<\infty$ such that for any positive $\varepsilon$ and $\delta$, and integer $r\geq 1$,
\begin{equation}\label{eq:LCexistence:1}
{\mathbb P}_{\sT,p_c}
\left(\exists x\colon |{\mathcal C}_{\sT}(x)|>\delta V^{2/3},
~{\mathcal C}_{\sT}(x)~\mbox{contains a long cycle of length}~\leq \varepsilon V^{1/3}\right) 
\leq \frac{C\varepsilon}{\delta} ,\
\end{equation}
and 
\begin{equation}\label{eq:LCexistence:2}
{\mathbb P}_{\sT,p_c}
\left(\exists \mbox{ a long cycle of length}~\geq \varepsilon^{-1} V^{1/3}\right)\leq C\varepsilon .\
\end{equation}
\end{theorem}

\medskip

We next study the number of long cycles. We start by defining what this is.
For a subgraph $G$ of the torus, we define $Y_G$ as the smallest $k$ for
which there exist edges $e_1,\ldots, e_k$ in $G$ such that
$G\setminus\{e_1,\ldots,e_k\}$ does not contain any long cycles.
For $\delta>0$, we define
    \[
     Y_\delta = \sum_{\mathcal C} Y_{\mathcal C} I(|\mathcal C| >\delta V^{2/3}) ,\
    \]
where the sum is over all open clusters ${\mathcal C}$ of the torus.
We prove the following theorems:

\begin{theorem}
\label{thm:Ytight}
Assume (\ref{prop2pf}).
There exists $C<\infty$ such that for all $\delta>0$ and integer $r\geq 1$,
    \[
    {\mathbb E}_{\sT,p_c}[Y_\delta] \leq C/\delta .\
    \]
In particular, the random variables $Y_\delta$ are tight.
\end{theorem}

\begin{theorem}[Non-trivial existence of long cycles]
\label{thmExisNCC}
Assume (\ref{prop2pf}).\\
(a) There exists $c>0$ such that for all integers $r\geq 1$,
\[
{\mathbb P}_{\sT, p_c}\left(\exists \mbox{ a long cycle of length }>c V^{1/3}\right)
> c .\
\]
(b) For any $\delta>0$ there exists $c>0$ such that for all integers $r\geq 1$,
\[
{\mathbb P}_{\sT, p_c}(Y_\delta = 0) > c .\
\]
In other words, with positive probability uniformly in $r$, there are no long cycles in clusters of size $>\delta V^{2/3}$.
\end{theorem}

\subsection{Related results on critical percolation}
\label{sec-extra}
In this section, we state a few results about critical Bernoulli percolation 
on $\Z^d$ and $\torus$ that are interesting in their own right,
for the ease of future reference.

\begin{theorem}[Connections inside balls]
\label{thm:ConnInsideBalls}
Assume (\ref{prop2pf}). There exists $C<\infty$ such that\\
(a) for all $n\geq 1$,
\[
\sum_{x\in \partial \Qn{n}}{\mathbb P}_{\sZ, p_c}(0\leftrightarrow x\mbox{ in } \Qn{n}) \leq C ,\
\]
(b) for all $\varepsilon>0$,
    \begin{equation}\label{eq:ConnInsideBalls(b)}
    \limsup_{n\to\infty}
    \sum_{x\in \partial \Qn{n}}
    {\mathbb P}_{\sZ, p_c}(0\stackrel{\leq \varepsilon n^2}\longleftrightarrow x\mbox{ in } \Qn{n})
    \leq C\sqrt{\varepsilon} .\
    \end{equation}
\end{theorem}

\begin{theorem}[Short connections]
\label{thm:ShortConn:r}
Assume (\ref{prop2pf}). There exists $C<\infty$ such that for any $z\in\Z^d$ and positive integers $r$ and $k$, 
\begin{equation}\label{eq:ShortConn:r}
\sum_{x\in\Z^d,~x\stackrel{r}\sim z,~|x|\geq \lfloor r/4\rfloor} 
{\mathbb P}_{\sZ, p_c}\left(0\stackrel{\leq k}\longleftrightarrow x\mbox{ in } \Z^d\right)
\leq 
\left\{
\begin{array}{ll}
\frac{Ck}{r^d} &\mbox{if } k\geq r^2 ,\\
Cr^{2-d}\cdot \e^{-\frac{r}{C\sqrt{k}}} &\mbox{if }k \leq r^2 ,\
\end{array}
\right.
\end{equation}
and for $k\geq r^2$, 
\begin{equation}\label{eq:ShortConn:3r}
\sum_{w\stackrel{r}\sim z,~|w|\geq 3r/4}\sum_{u,v}
{\mathbb P}_{\sZ, p_c}\left(0\stackrel{\leq k}\longleftrightarrow u\mbox{ in } \Z^d\right)
{\mathbb P}_{\sZ, p_c}\left(u\stackrel{\leq k}\longleftrightarrow v\mbox{ in } \Z^d\right)
{\mathbb P}_{\sZ, p_c}\left(v\stackrel{\leq k}\longleftrightarrow w\mbox{ in } \Z^d\right)
\leq \frac{Ck^3}{r^d} .\
\end{equation}
\end{theorem}

We note that in the special case where $r=1$, \eqref{eq:ShortConn:r} implies the result of \cite[Theorem~1.2(i)]{KN(IIC)} that 
	\begin{equation}
	\label{KN-ball}
	\sum_{x\in\Z^d} 
	{\mathbb P}_{\sZ, p_c}\left(0\stackrel{\leq k}\longleftrightarrow x\mbox{ in } \Z^d\right)
	\leq Ck.
	\end{equation}

\begin{theorem}[Torus two-point function]
\label{thmTwoPointTorus}
There exists $C<\infty$ such that
for all $x\in {\mathbb T}_r^d$,
    \begin{equation}\label{eq:torus2pf}
    \tau_{\sT,p_c}(x)\leq \tau_{\sZ,p_c}(x)+ C V^{-2/3},
    \end{equation}
and for all positive integers $n<r/2$, 
\begin{equation}\label{eq:torus2pf:n}
\sup_{x\in {\mathbb T}_r^d}
{\mathbb P}_{\sT,p_c}\left(0\leftrightarrow x \text{ by a path which visits }\partial\Qn{n}\right)
\leq Cn^{2-d} + C V^{-2/3}.
\end{equation}
\end{theorem}

\subsection{Discussion}
\label{sec-dis}
In this section, we compare our results to those for the Erd\H{o}s-R\'enyi random graph (ERRG),
as proved, for example, by Aldous in \cite{Aldo97}, and formulate some open problems.

\paragraph{Cycle structure on the Erd\H{o}s-R\'enyi random graph.}
We refer to \cite{Aldo97} for the extensive literature on the cycle structure of the ERRG. 
The ERRG is obtained by removing each edge of the complete graph $K_n$ independently 
with probability $p$. On the critical ERRG, there is no distinction between long and short cycles, 
and the number of cycles of a cluster equals the tree excess of the cluster,
i.e., the minimal number of edges one needs to remove from the cluster 
in order for it to be a tree.
 
On the ERRG, within the critical window $p=(1+\lambda n^{-1/3})/n$ for some $\lambda\in \R$, 
the number of cycles of large clusters converges in distribution to a Poisson random variable with a 
{\it random} parameter. This random parameter
can be described as the \emph{area} of the cluster exploration process. Since this parameter is
a bounded random variable, in particular, each large cluster has a tree excess
that converges in distribution, and the probability that the $i^{\rm th}$ largest cluster 
does not contain any cycle is strictly positive for any $i\geq 1$ fixed. 

Since there is just a finite number of clusters
of size at least $\delta n^{2/3}$, this immediately implies that the probability
that there are no connected components of size at least $\delta n^{2/3}$
containing cycles is strictly positive. Further, all cycles have macroscopic length.
Indeed, the largest connected components in the ERRG have diameter of the order
$n^{1/3}$, and the length of cycles (when they exist) is also of the same order of magnitude.
Cycles play a crucial role in describing the scaling limit of the largest
critical clusters on the ERRG, as identified in \cite{AddBroGol10}. Indeed,
clusters locally look like trees, with cycles creating shortcuts between the
different branches of the tree. Since cycles have a macroscopic length, these
shortcuts are also macroscopic and thus the scaling limit of large critical clusters
on the ERRG within the critical window is {\it not} a tree.

We conclude that the main features of the scaling limit of the critical
ERRG are (a) the largest critical clusters being of size $O(n^{2/3})$,
with a non-concentrated limit; 
(b) the largest critical clusters being close to trees with 
at most a finite number of macroscopic cycles; 
(c) a non-trivial probability that there exists large clusters having cycles.

\paragraph{Cycle structure on high-dimensional tori.}
We next investigate percolation on high-dimensional tori.
In this case, of the above features (a-c) of the scaling limit of the critical ERRG, 
the feature (a) has been investigated in \cite{BCHSS05a, BCHSS05b, HH,HH1},
we focus on features (b) and (c) here.  While our results clearly are not
as strong as on the ERRG, they do establish the non-triviality of 
the probability of existence of cycles in large clusters as well
as bounds on their length. Based on our results, we see that a natural split
exists between cycles containing $O(V^{1/3})$ vertices that are truly
macroscopic, and short cycles that basically remain within a cube without leaving its
boundary. The former are essential in describing the scaling limit,
the latter vanish in the scaling limit. There is no middle ground.

Our results provide yet another argument why the scaling limit of 
critical percolation on high-dimensional
tori should be related to that for the critical ERRG.
More precisely, the scaling limit of large critical clusters
on the ERRG within the critical window is, or is not, a tree, each with positive 
probability. We see the same for 
large critical clusters on the high-dimensional torus, where with high probability,
there are no long cycles containing $o(V^{1/3})$ vertices in large cluster.
Thus, all long cycles are macroscopic, and will thus change the scaling limit
of large critical clusters, in a similar way as they do on the ERRG.

\paragraph{Open probems.}
We complete this section by formulating a few open problems.
The first extension deals with the values of $p$ within the so-called \emph{scaling window.}
The results in \cite{HH, HH1}, in conjunction with those in
\cite{BCHSS05a, BCHSS05b}, show that when $p=p_c(1+\varepsilon)$, where $V^{1/3}|\varepsilon|$
remains uniformly bounded, the largest clusters obey similar scaling as for $p=p_c$.
An open problem is to show that Theorems \ref{thmNCCExponent}--\ref{thmExisNCC}
remain valid throughout the scaling window. 
Another extension is to more general
high-dimensional tori, for example, to percolation on the hypercube as 
studied in \cite{BCHSS06, HN12}. An open problem is to prove that Theorems \ref{thmNCCExponent}--\ref{thmExisNCC}
also hold on the hypercube, where a cycle is defined to be \emph{long} when the number of 
edges in it is at least the random walk mixing time $n\log{n}$. The random walk mixing time
plays a crucial role in \cite{HN12} to identify the supercitical behavior of percolation on 
the hypercube.

\paragraph{Organization of this paper.}
This paper is organized as follows. In Section \ref{sec-prel}, we collect some
preliminary results that we use in the proof. In Section \ref{sec-pfthmNCCExponent},
we prove Theorem~\ref{thmNCCExponent}. In Section \ref{sec-PfthmNCCExistence}, we prove
Theorem~\ref{thmNCCExistence}. In Section \ref{sec-PfthmNumNCC_and_thmExisNCC}, we
prove Theorems~\ref{thm:Ytight} and \ref{thmExisNCC}. 
In Section \ref{sec-pf-thmShortArms}, we prove Theorem~\ref{thm:ConnInsideBalls}, in Section~\ref{sec:ShortConn:r} we prove Theorem~\ref{thm:ShortConn:r},
 and in Section~\ref{sec:thmTwoPointTorus} we prove Theorem~\ref{thmTwoPointTorus}.

\section{Preliminary results}
\label{sec-prel}
\noindent
In this section we collect some results that we use in the proofs.
\subsection{Coupling clusters on the torus and the lattice}
\label{sec-coupling}

We say that two vertices $x$ and $y$ in $\Z^d$ are {\it $r$-equivalent}
and write $x\stackrel{r}\sim y$, if $y = x + rz$ for some $z\in \Z^d$.
We say that two edges $e_1 = \{x_1,y_1\}$ and
$e_2 = \{x_2,y_2\}$ in $\mathbb E^d$ are $r$-{\it equivalent} ($e_1\stackrel{r}\sim e_2$) if
$x_1\stackrel{r}\sim x_2$ and $y_1\stackrel{r}\sim y_2$, or if
$x_1\stackrel{r}\sim y_2$ and $y_1\stackrel{r}\sim x_2$.
In the proof of Theorem~\ref{thmNCCExponent}, we need an extension
of \cite[Proposition~2.1]{HH}:

\begin{proposition}[Coupling of clusters on torus and lattice]
\label{prCoupling}
Consider bond percolation on $\Z^d$ and on $\mathbb T_r^d$ with parameter $p\in[0,1]$.
There exists a coupling ${\mathbb P}_{\sZ,\sT,p}$ of ${\mathbb P}_{\sZ,p}$ and 
${\mathbb P}_{\sT,p}$ on the joint space of percolation on $\mathbb Z^d$ and $\mathbb T_r^d$ 
such that
the following properties are satisfied ${\mathbb P}_{\sZ,\sT,p}$-almost surely for all $k$:
\begin{itemize}\itemsep1pt
\item[(a)]
for all $x\in \mathbb T_r^d$,
\[
\{0\stackrel{\leq k}\longleftrightarrow x\mbox{ in }\mathbb T_r^d\} \subseteq 
\bigcup_{y\in \Z^d,y\stackrel{r}\sim x}\{0\stackrel{\leq k}\longleftrightarrow y~\mbox{in}~\Z^d\},
\]
\item[(b)]
for $x\stackrel{r}\sim y$, the event 
\[
\{0\stackrel{\leq k}\longleftrightarrow y~\mbox{in}~ \Z^d\}\setminus    
\{0\stackrel{\leq k}\longleftrightarrow x~\mbox{in}~ \mathbb T_r^d\}    
\]
implies that there exist distinct $r$-equivalent vertices $v_1$ and $v_2$ in $\Z^d$ and a vertex $z\in \Z^d$ such that
the following disjoint connections take place in $\Z^d$:
\[
\{0\stackrel{\leq k}\longleftrightarrow z\} \circ 
\{z\stackrel{\leq k}\longleftrightarrow v_1\} \circ 
\{z\stackrel{\leq k}\longleftrightarrow v_2\} \circ 
\{v_1 \stackrel{\leq k}\longleftrightarrow y\}.
\]
\end{itemize}
\end{proposition}

\cite[Proposition~2.1]{HH} is Proposition \ref{prCoupling} for $k=\infty$. 

\begin{proof}
Let $\varphi$ be the map from $\mathbb E_r^d$ to subsets of $\mathbb E^d$ defined by  
	\[
	\varphi(\{x,y\}) = \left\{\{x',y'\}\in\mathbb E^d~:~x'\stackrel{r}\sim x,~y'\stackrel{r}\sim y\right\}
	\qquad\mbox{for all }\{x,y\}\in\mathbb E_r^d .\
	\]
The sets $(\varphi(e))_{e\in\mathbb E_r^d}$ form the equivalence classes of the equivalence
relation $e\stackrel{r}\sim f$, where $e\stackrel{r}\sim f$ denotes that the endpoitns are 
$r$-equivalent. Note that for each $e\in\mathbb E_r^d$, $\varphi(e)\neq\varnothing$, 
and the sets $(\varphi(e))_{e\in\mathbb E_r^d}$ form a partition of $\mathbb E^d$. 
Let $\Phi$ be the map from $\{0,1\}^{\mathbb E_r^d}$ to $\{0,1\}^{\mathbb E^d}$ defined by 
	\[
	\Phi(\omega)_f = 
	\omega_e\qquad\quad \mbox{for all }\omega\in\{0,1\}^{\mathbb E_r^d},
	~e\in\mathbb E_r^d,~ f\in\varphi(e) .\
	\]
Informally, we prove the proposition by defining 
for each percolation configuration on $\mathbb T_r^d$
a certain ``unwrapping'' of the set of edges with an end-vertex in $\mathcal C_{\sT}(0)$ onto $\mathbb E^d$, 
and then revealing the status of 
the remaining edges of $\mathbb E^d$ by sampling from an independent percolation configuration on $\Z^d$. 
We achieve this by constructing an exploration of the edges of $\mathcal C_{\sT}(0)$ 
for each percolation configuration on $\mathbb T_r^d$. 
\medskip 

Let $\omega\in\{0,1\}^{\mathbb E_r^d}$. For $n\geq 0$, 
we will define the following sets of edges in $\mathbb E_r^d$ recursively in $n$: 
	\[
	\begin{array}{c}
	\text{$A_{\sT}(n)$ (active edges), \hskip1cm $O_{\sT}(n)$ (occupied edges),
	\hskip1cm
	$V_{\sT}(n)$ (vacant edges), \hskip0.5cm and}\\
	\text{$E_{\sT}(n) = O_{\sT}(n)\cup V_{\sT}(n)$ (explored edges);}
	\end{array}
	\]
and the following sets of edges in $\mathbb E^d$:
	\[
	\begin{array}{c}
	\text{$A_{\sZ}(n)$ (active edges), \hskip1cm
	$O_{\sZ}(n)$ (occupied edges), \hskip1cm
	$V_{\sZ}(n)$ (vacant edges),}\\
	\text{$G_{\sZ}(n)$ (ghost edges), \hskip0.5cm and \hskip0.5cm 
	$E_{\sZ}(n) = O_{\sZ}(n)\cup V_{\sZ}(n)\cup G_{\sZ}(n)$ (explored edges). 
	}
	\end{array}
	\]
We refer to this recursive procedure as the {\it exploration}.
\medskip

We initiate the exploration by taking $A_{\sZ}(0)$ to be 
the set of all edges that are neighbors of the origin in $\Z^d$, $E_{\sZ}(0) = E_{\sT}(0) = \varnothing$, and 
$A_{\sT}(0) = \{e\in\mathbb E_r^d~:~\varphi(e)\cap A_{\sZ}(0)\neq\varnothing\}$. 
Note that $A_{\sT}(0)$ is the set of all edges in $\mathbb E_r^d$ that are neighbors of the origin in $\mathbb T_r^d$. 

Let $n\geq 0$. We assume that the exploration is defined up to step $n$ and now define it at step $(n+1)$. 
If $A_{\sZ}(n)=\varnothing$, then we stop the exploration and write $T = n-1$. 
Otherwise, we take an edge $e\in A_{\sZ}(n)$ which is {\it closest to the origin in terms of the graph distance in $O_{\sZ}(n)$} 
(with ties broken in an arbitrary deterministic fashion).
We define $G_{\sZ}(n+1)=G_{\sZ}(n)\cup\{f\colon f\neq e,f\stackrel{r}\sim e\}$. 
To make the other updates, we consider two cases:
\begin{itemize}\itemsep0pt
\item[(a)] 
if $\Phi(\omega)_e = 1$, then we define $O_{\sZ}(n+1) = O_{\sZ}(n)\cup\{e\}$, $V_{\sZ}(n+1) = V_{\sZ}(n)$, 
and $A_{\sZ}(n+1)=A_{\sZ}(n)\cup\{f\colon f\sim e\}\setminus E_{\sZ}(n+1)$, 
where $e\sim f$ means that $e$ and $f$ share an end-vertex; 
\item[(b)] 
if $\Phi(\omega)_e = 0$, then we define $O_{\sZ}(n+1) = O_{\sZ}(n)$, $V_{\sZ}(n+1)=V_{\sZ}(n)\cup\{e\}$, and 
$A_{\sZ}(n+1) = A_{\sZ}(n)\setminus E_{\sZ}(n+1)$.
\end{itemize}
Finally, for $S\in\{A,O,V,E\}$, we define 
$S_{\sT}(n+1) = \{e\in\mathbb E_r^d~:~\varphi(e)\cap S_{\sZ}(n+1)\neq\varnothing\}$.

Note that in the exploration the status of each edge $f\in\mathbb E_r^d$ (i.e., the value of $\omega_f$) is 
checked at most once. Indeed, once an edge $e\in\varphi(f)$ is selected, all the remaining edges in its equivalence class $\varphi(f)$ are immediately declared ghost, and therefore, cannot become active anymore. 
In particular, for each $S\in\{A,O,V,E\}$, if $S_{\sZ}(n+1)\setminus S_{\sZ}(n)\neq\varnothing$, then 
$S_{\sT}(n+1)\setminus S_{\sT}(n)\neq\varnothing$.

Also note that the exploration eventually stops, since $T$ is at most the number of edges in $\mathbb E_r^d$. 
Moreover, the set $O_{\sT}(T)$ coincides with the set of open edges of $\mathcal C_{\sT}(0)$ in $\omega$, and
the set $V_{\sT}(T)$ is the set of closed edges of $\mathbb E_r^d$ in $\omega$ sharing an end-vertex with $\mathcal C_{\sT}(0)$. 
In other words, the exploration stops as soon as all the edges with an end-vertex in $\mathcal C_{\sT}(0)$ are explored. 
\medskip

To complete the construction of the coupling, we define,
for each $\omega\in\{0,1\}^{\mathbb E_r^d}$ and $\overline \omega\in\{0,1\}^{\mathbb E^d}$, the configuration 
$\widetilde \omega\in\{0,1\}^{\mathbb E^d}$ such that for all $e\in\mathbb E^d$, 
	\[
	\widetilde \omega_e = 
	\left\{
	\begin{array}{l}
	\text{$\Phi(\omega)_e$\hskip1cm if $e\in O_{\sZ}(T)\cup V_{\sZ}(T)$,}\\
	\text{$\overline \omega_e$ \hskip1.5cm otherwise.}
	\end{array}
	\right.
	\]
Note that for any $p\in[0,1]$, if $(\omega,\overline \omega)$ is sampled from $\mathbb P_{\sT,p}\otimes\mathbb P_{\sZ,p}$, 
then $\widetilde \omega$ gives a sample from $\mathbb P_{\sZ,p}$. 
Therefore, $(\omega,\widetilde\omega)$ gives us a coupling of $\mathbb P_{\sT,p}$ and $\mathbb P_{\sZ,p}$.
It remains to check that this coupling satisfies the properties defined in the statement of the proposition. 
\medskip

Property (a) is immediate from the construction. 
It remains to prove property (b).
Take $x\in\mathbb T_r^d$, $y\in\Z^d$, with $x\stackrel{r}\sim y$, and an integer $k\geq 0$. 
Let $\omega\in\{0,1\}^{\mathbb E_r^d}$ and $\overline\omega\in\{0,1\}^{\mathbb E^d}$ be 
such that 
	\[
	\widetilde\omega\in\{0\stackrel{\leq k}\longleftrightarrow y\mbox{ in }\Z^d\}
	\qquad\mbox{and}\qquad
	\omega\notin\{0\stackrel{\leq k}\longleftrightarrow x\mbox{ in }\mathbb T_r^d\} .\
	\]
We need to show that there exist distinct $r$-equivalent vertices $v_1$ and $v_2$ in $\Z^d$ and a vertex $z\in \Z^d$ such that
	\[
	\widetilde \omega\in 
	\{0\stackrel{\leq k}\longleftrightarrow z\} \circ 
	\{z\stackrel{\leq k}\longleftrightarrow v_1\} \circ 
	\{z\stackrel{\leq k}\longleftrightarrow v_2\} \circ 
	\{v_1 \stackrel{\leq k}\longleftrightarrow y\}.
	\]
We fix a shortest open path $\pi$ from $0$ to $y$ in $\Z^d$ in $\widetilde \omega$. 
By assumption, the length $k'$ of $\pi$ is at most $k$. 
Since $\omega\notin\{0\stackrel{\leq k}\longleftrightarrow x\mbox{ in }\mathbb T_r^d\}$, 
there exists an edge $e$ on this path such that $e\in G_{\sZ}(T)$.
Let $e$ be the first edge on $\pi$ from $G_{\sZ}(T)$ (counting from $0$ to $y$).
We denote by $(e_1,\ldots,e_{m-1})$ all the edges on the part of $\pi$ from $0$ to $e$, and 
by $(e_{m+1},\ldots, e_{k'})$ all the edges on the part of $\pi$ from $e$ to $y$. 
Since $\pi$ is chosen to be a shortest open path from $0$ to $y$, 
there exists $n$ such that 
$e_1,\ldots,e_{m-1}\in O_{\sZ}(n)$ and $e_{m+1},\ldots,e_{k'}\notin O_{\sZ}(n)$. 
Indeed, otherwise by the definition of the exploration, 
there would exist an open path from $0$ to $y$ which is strictly shorter than $\pi$.  

By construction, there exists an edge $f\in\Z^d$ such that $f\neq e$, $f\stackrel{r}\sim e$, 
and the origin is connected to one of the end vertices of $f$ by an open path $\pi'$ inside $O_{\sZ}(n)$.
In particular, the length of $\pi'$ is at most $k$.
We denote this vertex by $v_2$, and let $v_1$ be the end-vertex of $e$ which is $r$-equivalent to $v_2$.
Let $\pi_1$ be the part of $\pi$ from $0$ to $v_1$, and $\pi_2$ be the part of $\pi$ from $v_1$ to $y$. 

Note that $\pi_1$ and $\pi_2$ are edge-disjoint by definition. 
Moreover, since the edges of $\pi'$ are all in $O_{\sZ}(n)$, 
and the edges of $\pi_2$ are not in $O_{\sZ}(n)$, 
we deduce that $\pi'$ and $\pi_2$ are also edge-disjoint. 
We have thus shown that 
	\[
	\widetilde\omega \in 
	\bigcup_{v_1\neq v_2,~v_1\stackrel{r}\sim v_2}
	\{0\stackrel{\leq k}\longleftrightarrow v_1\mbox{ in }\Z^d,~
	0\stackrel{\leq k}\longleftrightarrow v_2\mbox{ in }\Z^d\}
	\circ\{v_1\stackrel{\leq k}\longleftrightarrow y~\mbox{in}~\Z^d\} .\
	\]
We finish the proof by observing that if $0$ is connected to $v_1$ and $v_2$ by open paths in $\Z^d$ 
of length at most $k$, then
there exists $z\in \Z^d$ such that the following edge-disjoint open paths (each of length at most $k$) exist in $\Z^d$:
from $0$ to $z$,  from $z$ to $v_1$, and from $z$ to $v_2$.
\end{proof}

\medskip

From now on, we only consider the probability measure ${\mathbb P}_{\sZ,\sT,p}$ defined in Proposition~\ref{prCoupling} and
\begin{center}
{\it in the remainder of the paper, we write ${\mathbb P}_p$ for ${\mathbb P}_{\sZ,\sT,p}$.}
\end{center}
In particular, $\mathbb P_p(E) = {\mathbb P}_{\sZ,p}(E)$ for $E\in {\mathcal F}_{\sZ,p}$, and
$\mathbb P_p(E) = {\mathbb P}_{\sT,p}(E)$ for $E\in {\mathcal F}_{\sT,p}$ (see Section~\ref{sec-coupling} for notation), and
we always assume without mentioning the coupling from Proposition~\ref{prCoupling}
when we consider events from ${\mathcal F}_{\sZ,p}$ and ${\mathcal F}_{\sT,p}$ simultaneously.

\subsection{Previous results}
\label{sec-pre-res}

In the next theorem, we summarize a number of results on high-dimensional percolation
on $\Z^d$ that we will often use in the proofs in this paper. 

We start by introducing balls in the intrinsic distance. 
For $x\in \Z^d$ and $k\geq 0$, 
we write
 $B_{\sZ,k}(x)=\{y\colon x\stackrel{\leq k}{\conn} y \text{ in }\Z^d\}$ to denote all vertices at graph distance 
at most $k$ from $x$ in ${\mathcal C}_{\sZ}(x).$ Similarly, for $x\in \torus$ and $k\geq 0$, 
we write
$B_{\sT,k}(x)=\{y\colon x\stackrel{\leq k}{\conn} y \text{ in }\torus\}$ to denote all vertices at graph distance 
at most $k$ from $x$ in ${\mathcal C}_{\sT}(x).$
Further, we let 
$\partial B_{\sZ,k}(x)=B_{\sZ,k}(x)\setminus B_{\sZ,k-1}(x)$ and
$\partial B_{\sT,k}(x)=B_{\sT,k}(x)\setminus B_{\sT,k-1}(x)$ denote the 
vertices at graph distance precisely equal to $k$ from $x$.
Let $G$ be a subgraph of $\Z^d$ or $\torus$, respectively.
We define $B_{\sZ,k}^{\sss G}(x)$, $\partial B_{\sZ,k}^{\sss G}(x)$, $B_{\sT,k}^{\sss G}(x)$ and 
$\partial B_{\sT,k}^{\sss G}(x)$ in the same way as
$B_{\sZ,k}(x)$, $\partial B_{\sZ,k}(x)$, $B_{\sT,k}(x)$ and 
$\partial B_{\sT,k}(x)$, except that we are now only allowed
to use edges from $G$.

\begin{theorem}[Critical behavior of high-dimensional percolation]
\label{thmHighD}
Assume (\ref{prop2pf}). There exist $c>0$ and $C<\infty$ such that:
\begin{itemize}
\item[(i)]
For all $x\in\torus$, $y\in\Z^d$, and positive integer $k$, 
	\begin{equation}
	\label{KN-intrinsic}
	\sup_{G}{\mathbb P}_{p_c}\left(\partial B_{\sZ,k}^{\sss G}(y)\neq \varnothing\right)
	\leq C/k, \qquad
	\sup_{G}{\mathbb P}_{p_c}\left(\partial B_{\sT,k}^{\sss G}(x)\neq \varnothing\right)
	\leq C/k.
	\end{equation}
\item[(ii)] 
For all positive integers $n$, 
\begin{equation}
\label{prop1Arm}
c n^{-2}\leq {\mathbb P}_{p_c}\left(0\leftrightarrow \partial \Qn{n}~\mbox{in}~{\mathbb Z}^d\right) \leq C n^{-2},
\end{equation}
\item[(iii)]
\begin{equation}
\label{propSizeCn}
c n^2 \leq \sum_{x\in \Qn{n}}\tau_{\sZ,p_c}(x)\leq C n^2,
\end{equation}
and for any given $z\in\mathbb Z^d$ and a positive integer $r$ with $r\leq n$,
\begin{equation}
\label{propSizeCnr}
\frac{cn^2}{r^d} \leq \sum_{x\in \Qn{n},x\stackrel{r}\sim z, |x|\geq r/8}\tau_{\sZ,p_c}(x)\leq \frac{Cn^2}{r^d}.
\end{equation}
\item[(iv)]
For any $z\in\mathbb Z^d$,
\begin{equation}\label{propTriangle}
\sum_{x,y\in\mathbb Z^d}\tau_{\sZ,p_c}(0,x)\tau_{\sZ,p_c}(x,y)\tau_{\sZ,p_c}(y,z) \leq C|z|^{6-d}.
\end{equation}
\end{itemize}
\end{theorem}
\begin{proof}
The first statement is \cite[Theorem~1.2(ii)]{KN(IIC)},
and its adaptation to the torus in \cite[Verification of Theorem 4.1(b)]{HH1}.
Statement (ii) is \cite[Theorem~1]{KN}. 
Statements (iii) and (iv) easily follow from (\ref{prop2pf}).
\end{proof}

The next theorem gives an upper bound on ${\mathbb E}_{\sT,p_c}|\mathcal C(0)|$, which is used often in the proofs:
\begin{theorem}[Expected critical cluster size on torus]
\label{thm:expCT}
There exists $C<\infty$ such that for all $r\geq 1$,
\begin{equation}
\label{eq:expCT}
{\mathbb E}_{\sT,p_c} |\mathcal C(0)| \leq CV^{1/3} .\
\end{equation}
\end{theorem}
\begin{proof}
The statement follows from \cite[(1.6)]{HH} and \cite[Theorem 1.6(iii)]{BCHSS05b}.
(Alternatively, it follows from \eqref{eq:torus2pf} and \eqref{propSizeCn}.)
\end{proof}

\medskip

\section{Proof of Theorem~\ref{thmNCCExponent}}
\label{sec-pfthmNCCExponent}

In this section, we prove Theorem~\ref{thmNCCExponent} subject to Theorems \ref{thm:ConnInsideBalls}--\ref{thmTwoPointTorus}.
We give the proof of the upper bound in Proposition~\ref{prNCCExponent1}
and the proof of the lower bound in Proposition~\ref{prNCCExponent2}.
The results of these propositions are more general than the result of
Theorem~\ref{thmNCCExponent}, and also give bounds on the probability
that the origin is in a long cycle with small length.
In particular, the upper bound on the probability of such an event will be
used in the proof of Theorem~\ref{thmNCCExistence} to
show that large open clusters do not contain long cycles with few edges.

Recall that $\LC{k}$ denotes the event that the origin is in an open long
cycle of length at most $k$. We prove the following bounds:

\begin{proposition}[Upper bound on long cycles]
\label{prNCCExponent1}
Assume (\ref{prop2pf}).
There exists $C<\infty$ such that
    \[
    {\mathbb P}_{p_c}(0~\mbox{is in an open long cycle}) \leq CV^{-2/3},
    \]
and
    \[
    {\mathbb P}_{p_c}(\LC{k})
    \leq Ck/V.
    \]
\end{proposition}

\begin{proof}
We begin by proving the second statement of the proposition. Let $k\geq 1$. 
By the definition of a long cycle, if $\LC{k}$ occurs, then 
there exists $x\in\partial \Qn{r/4}$ such that 
the following connections occur disjointly:
	\[
	\{0\leftrightarrow x\mbox{ in }\Qn{r/4}\}\circ
	\{0\stackrel{\leq k}\longleftrightarrow x\mbox{ in }\mathbb T_r^d\} .\
	\]
By the BK inequality,
	\[
	\mathbb P_{p_c}(\LC{k})
	\leq 
	\sum_{x\in\partial \Qn{r/4}}
	\mathbb P_{p_c}\left(0\leftrightarrow x\mbox{ in }\Qn{r/4}\right)
	\mathbb P_{p_c}\left(0\stackrel{\leq k}\longleftrightarrow x\mbox{ in }\mathbb T_r^d\right) .\
	\]
By property (a) of Proposition~\ref{prCoupling} and \eqref{eq:ShortConn:r}, we obtain 
that for any $x\in\partial \Qn{r/4}$, 
	\[
	\mathbb P_{p_c}\left(0\stackrel{\leq k}\longleftrightarrow x\mbox{ in }\mathbb T_r^d\right)
	\leq 
	\sum_{y\in\Z^d, y\stackrel{r}\sim x}
	\mathbb P_{p_c}\left(0\stackrel{\leq k}\longleftrightarrow y\mbox{ in }\Z^d\right)
	\leq 
	C_1 k/V.\
	\]
Combining the above inequalities and using Theorem~\ref{thm:ConnInsideBalls}(a), we arrive at 
the second statement of the proposition. 

\medskip

We proceed with the proof of the first statement. 
As in the proof of the second statement,
	\[
	\mathbb P_{p_c}(0~\mbox{is in an open long cycle})
	\leq 
	\sum_{x\in\partial \Qn{r/4}}
	\mathbb P_{p_c}\left(0\leftrightarrow x\mbox{ in }\Qn{r/4}\right)
	\mathbb P_{p_c}\left(0\leftrightarrow x\mbox{ in }\mathbb T_r^d\right) .\
	\]
By \eqref{eq:torus2pf} in Theorem \ref{thmTwoPointTorus} and \eqref{prop2pf}, for any $x\in\partial\Qn{r/4}$, 
	\[
	\mathbb P_{p_c}\left(0\leftrightarrow x\mbox{ in }\mathbb T_r^d\right) 
	\leq C_2 r^{2-d} + C_2 V^{-2/3}
	\leq 2C_2 V^{-2/3} ,\
	\]
where the last inequality holds for $d>6$. 
Putting the bounds together and using Theorem~\ref{thm:ConnInsideBalls}(a), we obtain the first statement of the proposition. 
\end{proof}

\medskip

\begin{proposition}[Lower-bound on long cycles]
\label{prNCCExponent2}
Assume (\ref{prop2pf}).
There exist constants $c,\varepsilon>0$ and $K<\infty$ such that 
	\[
	{\mathbb P}_{p_c}(0~\mbox{is in an open long cycle}) \geq cV^{-2/3} ,\
	\]
and for any $k\in[Kr^2,\varepsilon V^{1/3}]$,
	\[
	{\mathbb P}_{p_c}(\LC{k}) \geq ck/V .\
	\]
\end{proposition}

\begin{proof}
The first statement immediately follows from the second one for $k=\varepsilon V^{1/3}$.
We now prove the second statement.
Let $R$ be a large positive integer. Let $K$ be a large positive number, and $\varepsilon$ a small positive number.  
The precise choice of these numbers will be made later in the proof.
Let $k$ be an integer with $k\in[Kr^2,\varepsilon V^{1/3}]$. 

First of all, note that it suffices to prove the result for $r > 16R$.
Indeed, once we fix $R$, the result for $r\in [1,16R]$ will follow by adjusting the constant $c$.
Therefore, throughout the proof we assume that $r>16R$.
The proof consists of several steps.
\\
{\it Step~1.} For any $x\in {\mathbb Z}^d$, let $u(x)$ be the vertex in ${\mathbb Z}^d$ 
with coordinates $(x_1,\ldots,x_{d-1},x_d+R)$.
In particular, $u(x)\in\partial \Qn{R}(x)$.
For $x\in{\mathbb Z}^d$, consider the event
    \[
    A_x = A_x(k,R) = \{0\stackrel{\leq k}\longleftrightarrow u(x)\}.
    \]
Let
    \[
    N(A) = \left|\{x\in{\mathbb Z}^d\colon x\stackrel{r}\sim 0,~x\neq 0,~A_x~\mbox{occurs}\}\right|.
    \]
We use the second moment method to show that
    \[
    {\mathbb P}_{p_c}(N(A)\neq 0) \geq ck/V,
    \]
for some constant $c$ that may depend on $K$, but not on $r$, $k$ or $R$.

We first show that there exists a constant $C_3=C_3(K)$ such that
    \begin{equation}\label{eq:ENAlowerbound}
    {\mathbb E}_{p_c}N(A) \geq  C_3 k/V.
    \end{equation}
We write 
	\begin{equation}\label{eq:ENAlowerbound:1}
	{\mathbb E}_{p_c}N(A)
	\geq
	\sum_{x\stackrel{r}\sim 0,~ r/16\leq |x|\leq \sqrt{k/K}}
	{\mathbb P}_{p_c}(0\stackrel{\leq 	k}\longleftrightarrow u(x)) .\
	\end{equation}
Similarly to the proof of \cite[Theorem~1.3(i)]{KN(IIC)}, one can show that if $K = K(d)$ is chosen large enough, then 
for any $r/16\leq |x|\leq \sqrt{k/K}$, 
\begin{equation}\label{eq:2pf:k}
{\mathbb P}_{p_c}(0\stackrel{\leq k}\longleftrightarrow u(x)) \geq \frac 12 
{\mathbb P}_{p_c}(0\leftrightarrow u(x)) .\
\end{equation}
Inequality \eqref{eq:ENAlowerbound} now follows from \eqref{eq:ENAlowerbound:1}, \eqref{eq:2pf:k}, 
the assumptions $r>16R$ and $k\geq Kr^2$, 
and the lower bound in \eqref{propSizeCnr}, where $C_3=C_3(K)=c_3/K$ for some 
$c_3>0$ independent of all other parameters. 
From this moment onwards, the large integer $K$ remains unchanged. 
\medskip

Next, we bound the second moment of $N(A)$. Let ${\sum}'$ be the sum over all
distinct $x,y\in \Z^d$ such that $x,y\neq 0$ and $x,y\stackrel{r}\sim 0$.
We obtain
    \[
    {\mathbb E}_{p_c}N(A)^2
    \leq {\mathbb E}_{p_c}N(A) + 
{\sum}'{\mathbb P}_{p_c}\left(0\stackrel{\leq k}\longleftrightarrow u(x),~ 0\stackrel{\leq k}\longleftrightarrow u(y)\right).
    \]
Note that if $0\stackrel{\leq k}\longleftrightarrow u(x)$ and 
$0\stackrel{\leq k}\longleftrightarrow u(y)$, then there exists $z\in \Z^d$ such that the following
open paths occur disjointly: 
$0\stackrel{\leq k}\longleftrightarrow z$, $z\stackrel{\leq k}\longleftrightarrow u(x)$ and $z\stackrel{\leq k}\longleftrightarrow u(y)$.
Therefore, the BK inequality implies
    \[
    {\mathbb E}_{p_c}N(A)^2
    \leq {\mathbb E}_{p_c}N(A) +
    {\sum}'
    \sum_{z\in \Z^d} 
{\mathbb P}_{p_c}\left(0\stackrel{\leq k}\longleftrightarrow z\right)
{\mathbb P}_{p_c}\left(z\stackrel{\leq k}\longleftrightarrow u(x)\right)
{\mathbb P}_{p_c}\left(z\stackrel{\leq k}\longleftrightarrow u(y)\right) .\
    \]
Let ${\sum}''$ be the sum over all pairwise distinct $x,y,z\in\Z^d$ such that
$x\stackrel{r}\sim y\stackrel{r}\sim z$.
By translation invariance and the fact that $u(x) - z = u(x-z)$, we have 
    \[
    {\mathbb E}_{p_c}N(A)^2
    \leq {\mathbb E}_{p_c}N(A) + 
{\sum}''
{\mathbb P}_{p_c}\left(0\stackrel{\leq k}\longleftrightarrow z\right)
{\mathbb P}_{p_c}\left(0\stackrel{\leq k}\longleftrightarrow u(x)\right)
{\mathbb P}_{p_c}\left(0\stackrel{\leq k}\longleftrightarrow u(y)\right) .\
    \]
Since $x$, $y$ and $z$ are distinct and $r$-equivalent, at least two of them are at distance at least $r/2$ from the origin.
Therefore, the above sum is at most
    \begin{multline*}
    {\mathbb E}_{p_c}N(A) +
    \sum_{x\stackrel{r}\sim y\stackrel{r}\sim z;~ |x|,|y|\geq r/2}
    \left[
{\mathbb P}_{p_c}\left(0\stackrel{\leq k}\longleftrightarrow z\right)
{\mathbb P}_{p_c}\left(0\stackrel{\leq k}\longleftrightarrow u(x)\right)
{\mathbb P}_{p_c}\left(0\stackrel{\leq k}\longleftrightarrow u(y)\right) 
\right. \\
\left.
+ 
2~
{\mathbb P}_{p_c}\left(0\stackrel{\leq k}\longleftrightarrow x\right)
{\mathbb P}_{p_c}\left(0\stackrel{\leq k}\longleftrightarrow u(y)\right)
{\mathbb P}_{p_c}\left(0\stackrel{\leq k}\longleftrightarrow u(z)\right) 
    \right] .\
    \end{multline*}
Remember that we assume that $r>16R$. In particular, $|u(x)| \geq r/4$ when $|x|\geq r/2$.
Applying \eqref{eq:ShortConn:r} consequently to the sums over $x$, $y$, and then \eqref{KN-ball} to the sum over $z$, 
and then using the assumption $k\leq \varepsilon V^{1/3}$, we obtain that 
there exists $\varepsilon_0 = \varepsilon_0(d)>0$ such that for any $\varepsilon<\varepsilon_0$, 
\begin{equation}\label{eq:NA2auxbound}
{\mathbb E}_{p_c}N(A)^2
\leq 
{\mathbb E}_{p_c}N(A) +
C_4 k\cdot\left(\frac kV\right)^2
\leq {\mathbb E}_{p_c}N(A) +C_4 \varepsilon^2V^{-1/3}\cdot \frac kV 
\leq 2{\mathbb E}_{p_c}N(A) .\ 
\end{equation}
A second moment estimate, using \eqref{eq:ENAlowerbound} and \eqref{eq:NA2auxbound}, yields
    \begin{equation}\label{eq:NAsecondmoment}
    {\mathbb P}_{p_c}(N(A)\neq 0) \geq \frac{({\mathbb E}_{p_c}N(A))^2}{{\mathbb E}_{p_c}N(A)^2} \geq \frac{C_3 k}{2V} .\
    \end{equation}
\medskip

\noindent
{\it Step~2.} Consider the event
    \[
E = E(r,k,R) = 
\{0\stackrel{\leq k}\longleftrightarrow u(0)~\mbox{in}~{\mathbb T}_r^d~
\mbox{by an open path which visits }\partial \Qn{r/2}\}.
    \]
We show that for small enough $\varepsilon$ and large enough $R$,
    \[
    {\mathbb P}_{p_c}(E) \geq \frac 12 {\mathbb P}_{p_c}(N(A)\neq 0).
    \]
Since ${\mathbb P}_{p_c}(E) \geq {\mathbb P}_{p_c}(N(A)\neq 0) - {\mathbb P}_{p_c}(\{N(A)\neq 0\}\setminus E)$, we should show that
\begin{equation}\label{eq:NAauxbound}
{\mathbb P}_{p_c}(\{N(A)\neq 0\}\setminus E) \leq \frac 12 {\mathbb P}_{p_c}(N(A)\neq 0) ,\
\end{equation}
when $\varepsilon$ is chosen small enough and $R$ large enough.

If $N(A) \neq 0$ and $E$ does not occur, then according to Proposition~\ref{prCoupling},
there exist $x\in \Z^d$ with $x\stackrel{r}\sim 0$ and $x\neq 0$, a vertex $z\in \Z^d$,
and distinct vertices $v_1$ and $v_2$ in $\Z^d$ with $v_1\stackrel{r}\sim v_2$ such that
the following disjoint connections take place in $\Z^d$:
    \[
    \{0\stackrel{\leq k}\longleftrightarrow z\}\circ
\{z\stackrel{\leq k}\longleftrightarrow v_1\}\circ
\{z\stackrel{\leq k}\longleftrightarrow v_2\}\circ
\{v_1\stackrel{\leq k}\longleftrightarrow u(x)\}.
    \]
By the BK inequality, the probability of the event $\{N(A)\neq 0\}\setminus E$ is bounded from above by
\begin{equation}\label{EqNAminusCK}
	\sum_{x\stackrel{r}\sim 0, x\neq 0}\sum_{v_1\stackrel{r}\sim v_2, v_1\neq v_2} \sum_{z}
	{\mathbb P}_{p_c}\left(0\stackrel{\leq k}\longleftrightarrow z\right)
	{\mathbb P}_{p_c}\left(z\stackrel{\leq k}\longleftrightarrow v_1\right)
	{\mathbb P}_{p_c}\left(z\stackrel{\leq k}\longleftrightarrow v_2\right)
	{\mathbb P}_{p_c}\left(v_1\stackrel{\leq k}\longleftrightarrow u(x)\right) .\
	\end{equation}
Note that, since $v_1$ and $v_2$ are distinct and $r$-equivalent, either $|v_1-z|\geq r/2$ or $|v_2 - z|\geq r/2$.
Assume first that $|v_2 - z|\geq r/2$. It follows from \eqref{eq:ShortConn:r} that for any $v_1$ and $z$ fixed,
	\[
	\sum_{v_2\in \Z^d,~ v_2\stackrel{r}\sim v_1,~ |v_2-z| \geq r/2}
	{\mathbb P}_{p_c}\left(z\stackrel{\leq k}\longleftrightarrow v_2\right)
	\leq C_5 k/V,
    \]
where $C_5$ does not depend on $k$, $v_1$ or $z$.
On the other hand, by \eqref{eq:ShortConn:3r} and using the fact that $k\leq \varepsilon V^{1/3}$,
	\begin{equation}\label{EqNAminusCK1}
	\sum_{x\stackrel{r}\sim 0, x\neq 0}\sum_{v_1,z}
	{\mathbb P}_{p_c}\left(0\stackrel{\leq k}\longleftrightarrow z\right)
	{\mathbb P}_{p_c}\left(z\stackrel{\leq k}\longleftrightarrow v_1\right)
	{\mathbb P}_{p_c}\left(v_1\stackrel{\leq k}\longleftrightarrow u(x)\right) 
	\leq \frac{C_6 k^3}{V}
	\leq C_6 \varepsilon^3 .\
	\end{equation}
Therefore, for $v_2$ and $z$ with $|v_2-z|\geq r/2$, the sum \eqref{EqNAminusCK} is,
uniformly in $R$, bounded from above by
    \[
    C_5C_6\varepsilon^3k/V.
    \]
Next, consider the sum (\ref{EqNAminusCK}) in the case $|v_1-z|\geq r/2$.
By translation invariance, the sum (\ref{EqNAminusCK}) equals
    \[
    \sum_{x\stackrel{r}\sim 0,~ x\neq 0}\sum_{v_1\stackrel{r}\sim v_2,~ v_1\neq v_2} \sum_{z}
{\mathbb P}_{p_c}\left(0\stackrel{\leq k}\longleftrightarrow z\right)
{\mathbb P}_{p_c}\left(z\stackrel{\leq k}\longleftrightarrow v_1\right)
{\mathbb P}_{p_c}\left(z\stackrel{\leq k}\longleftrightarrow v_2\right)
{\mathbb P}_{p_c}\left(v_2\stackrel{\leq k}\longleftrightarrow u(x)+(v_2-v_1)\right) .\
    \]
By the definition of $u(x)$, the translation of $u(x)$ by $(v_2-v_1)$ equals $u(x + (v_2-v_1))$.
Since $v_1\stackrel{r}\sim v_2$, the translation of $x$ ($r$-equivalent to $0$) by $(v_2-v_1)$ is still $r$-equivalent to $0$.
However, note that it is possible that $x + (v_2-v_1) = 0$.
These observations imply that the above sum is bounded from above by
    \[
    \sum_{x\stackrel{r}\sim 0}\sum_{v_1\stackrel{r}\sim v_2, v_1\neq v_2} \sum_{z}
{\mathbb P}_{p_c}\left(0\stackrel{\leq k}\longleftrightarrow z\right)
{\mathbb P}_{p_c}\left(z\stackrel{\leq k}\longleftrightarrow v_1\right)
{\mathbb P}_{p_c}\left(z\stackrel{\leq k}\longleftrightarrow v_2\right)
{\mathbb P}_{p_c}\left(v_2\stackrel{\leq k}\longleftrightarrow u(x)\right) .\
    \]
Since we only consider the above sum in the case $z$ and $v_1$ satisfy $|v_1-z|\geq r/2$, we obtain as before that 
for any given $z$ and $v_2$, 
    \[
\sum_{v_1\in \Z^d,~ v_1\stackrel{r}\sim v_2,~ |v_1-z| \geq r/2}
{\mathbb P}_{p_c}\left(z\stackrel{\leq k}\longleftrightarrow v_1\right)
\leq C_5 k/V .\
    \]
It remains to bound the sum
    \[
    \sum_{z,v_2}\sum_{x\stackrel{r}\sim 0}
{\mathbb P}_{p_c}\left(0\stackrel{\leq k}\longleftrightarrow z\right)
{\mathbb P}_{p_c}\left(z\stackrel{\leq k}\longleftrightarrow v_2\right)
{\mathbb P}_{p_c}\left(v_2\stackrel{\leq k}\longleftrightarrow u(x)\right) .\
    \]
There are two cases depending on whether $x \neq 0$ or $x = 0$.
The case $x \neq 0$ can be considered similarly to (\ref{EqNAminusCK1}),
so the above sum is bounded from above by $C_6\varepsilon^3$ in this case.

It remains to consider the case $x = 0$. In this case $|u(x)| = R$, and 
we simply bound the above sum by 
\[
\sum_{z,v_2}
{\mathbb P}_{p_c}\left(0\leftrightarrow z\right)
{\mathbb P}_{p_c}\left(z\leftrightarrow v_2\right)
{\mathbb P}_{p_c}\left(v_2\leftrightarrow u(0)\right) ,\
\]
which is bounded from above by $C_7 R^{6-d}$ by \eqref{propTriangle}, where $C_7$ is independent of $R$. 
Therefore, for $v_1$ and $z$ with $|v_1-z|\geq r/2$, the sum (\ref{EqNAminusCK}) is bounded from above by
    \[
    C_5(2C_6\varepsilon^3 + C_7R^{6-d})k/V.
    \]
Now recalling \eqref{eq:NAsecondmoment}, we take $R$ large and $\varepsilon$ small so that \eqref{eq:NAauxbound} holds.
We obtain from \eqref{eq:NAsecondmoment} and \eqref{eq:NAauxbound} that
    \[
    {\mathbb P}_{p_c}(E) \geq C_3 k/(4V) .\
    \]
{\it Step~3.} We now show that there exists $C_8 = C_8(R)>0$ such that
    \[
    {\mathbb P}_{p_c}(\LC{k+R^d}) \geq C_8 {\mathbb P}_{p_c}(E).
    \]
This follows from a local modification argument as follows.
Note that if $E$ occurs, there exist $z$ and $z'$ on $\partial \Qn{R}$ such that 
$z$ is connected to $z'$ by a path in ${\mathbb T}_r^d\setminus \Qn{R}$ of length at most $k$
which visits $\partial \Qn{r/2}$. 
We can therefore modify the configuration of bonds inside $\Qn{R}$ to make sure that
$\{0\leftrightarrow z~\mbox{in}~\Qn{R}\}\circ \{0\leftrightarrow z'~\mbox{in}~\Qn{R}\}$,
which implies that the origin is in a long cycle of length at most $k+R^d$. 
Since there are only finitely many edges in $\Qn{R}$, the above inequality follows.

\medskip

We can now complete the proof of Proposition~\ref{prNCCExponent2}.
We pick $K=K(d)$ so that \eqref{eq:ENAlowerbound} holds for all $r >16R$.
We then pick $\varepsilon=\varepsilon(d)$ and $R=R(d)$ to satisfy \eqref{eq:NA2auxbound} and \eqref{eq:NAauxbound}.
It follows from Steps~2 and 3 of the proof that for all $r>16R$ and $k\in [Kr^2+R^d,\varepsilon V^{1/3}]$,
    \[
    {\mathbb P}_{p_c}(\LC{k}) \geq C_8\frac{C_3(k-R^d)}{4V} \geq C_9k/V.
    \]
Finally, we adjust the constant $C_9$ so that the result remained valid for $r\leq 16R$.
\end{proof}

\section{Proof of Theorem~\ref{thmNCCExistence}}
\label{sec-PfthmNCCExistence}

\begin{proof}[Proof of \eqref{eq:LCexistence:1}]
Let $\varepsilon>0$ and $\delta>0$. 
Let $M$ be the number of vertices $x\in\torus$ such that $|\mathcal C_{\sT}(x)|\geq \delta V^{2/3}$ and 
$\mathcal C_{\sT}(x)$ contains a long cycle of length at most $\varepsilon V^{1/3}$. 
We need to show that there exists $C<\infty$ such that 
	\[
	{\mathbb P}_{\sT,p_c}\left(M\neq 0\right) \leq C\varepsilon/\delta .\
	\]
By the definition of $M$ and the Markov inequality,
	\[
	{\mathbb P}_{\sT,p_c}\left(M\neq 0\right) = {\mathbb P}_{\sT,p_c}\left(M\geq \delta V^{2/3}\right)
	\leq \frac{\mathbb E_{\sT,p_c}[M]}{\delta V^{2/3}} .\
	\]
By translation invariance,
	\[
	\mathbb E_{\sT,p_c}[M]
	\leq V\cdot 
	{\mathbb P}_{\sT,p_c}\left(\mathcal C_{\sT}(0)\mbox{ contains a long cycle of 
	length at most }\varepsilon V^{1/3}\right) .\
	\]
Note that if $\mathcal C_{\sT}(0)$ contains a long cycle of length at most $\varepsilon V^{1/3}$, then 
there exists $z\in \mathcal C_{\sT}(0)$ such that the following events occur disjointly:
	\[
	\{0\leftrightarrow z\mbox{ in }\torus\}\circ \LC{\varepsilon V^{1/3}} .\
	\]
Therefore, application of the BK inequality gives that
	\[
	{\mathbb P}_{\sT,p_c}\left(\mathcal C_{\sT}(0)\mbox{ contains a long cycle of 
	length at most }\varepsilon V^{1/3}\right)
	\leq
	\mathbb E_{\sT,p_c}|\mathcal C_{\sT}(0)|\cdot \mathbb P_{\sT,p_c}(\LC{\varepsilon V^{1/3}}) .\
	\]
Using \eqref{eq:expCT} and Proposition~\ref{prNCCExponent1}, and putting all the bounds together gives \eqref{eq:LCexistence:1}. 
\end{proof}

\medskip

\begin{proof}[Proof of \eqref{eq:LCexistence:2}]
This is a simple consequence of Theorem \ref{thmNCCExponent}. 
Indeed, if there exists a long cycle of length at least
$\varepsilon^{-1} V^{1/3}$, then the number of vertices in long cycles is at least
$(2d \varepsilon)^{-1} V^{1/3}$. Denote the number of vertices in long cycles by $M$. 
Then, by the Markov inequality and Theorem \ref{thmNCCExponent},
\begin{multline*}
{\mathbb P}_{p_c}
\left(\exists \mbox{ a long cycle of length}~\geq \varepsilon^{-1} V^{1/3}\right)
\leq {\mathbb P}_{p_c}(M\geq (2d \varepsilon)^{-1} V^{1/3})
\leq 2d\varepsilon V^{-1/3}
{\mathbb E}_{p_c}[M]\\
=
2d\varepsilon V^{2/3}
{\mathbb P}_{p_c}
\left(\mbox{$0$ is in a long cycle}\right)
\leq C_1 \varepsilon .
\end{multline*}
\end{proof}

\section{Proof of Theorems~\ref{thm:Ytight} and \ref{thmExisNCC}}
\label{sec-PfthmNumNCC_and_thmExisNCC}

\subsection{Proof of Theorem~\ref{thm:Ytight}}\label{sec:Ytight}
Let $\mathcal I$ be the set of $z\in \mathcal C_{\sT}(0)$ such that
$\{0\leftrightarrow z\}\circ\{z\mbox { is in a long cycle}\}$.
Theorem~\ref{thm:Ytight} follows from Lemma \ref{l:YIJ}:
\begin{lemma}[Bound on the number of long cycles]
\label{l:YIJ}
    \[
    Y_{\mathcal C_{\sT}(0)} \leq 2d~ |\mathcal I| .\
    \]
Moreover, when (\ref{prop2pf}) holds, there
exists a finite constant $C$ such that
    \[
    {\mathbb E}_{p_c}|\mathcal I| \leq CV^{-1/3} .\
    \]
\end{lemma}

Before we prove Lemma \ref{l:YIJ}, we show how to
use it to complete the proof of Theorem~\ref{thm:Ytight}:
    \[
     {\mathbb E}_{p_c} \left[Y_\delta\right]
     =
     \sum_x {\mathbb E}_{p_c}\left[\frac{1}{\mathcal C_{\sT}(x)} Y_{\mathcal C_{\sT}(x)} I(|\mathcal C_{\sT}(x)| >\delta V^{2/3})\right]
     \leq
     (\delta V^{2/3})^{-1} V {\mathbb E}_{p_c} \left[Y_{\mathcal C_{\sT}(0)}\right]
     \leq
     C_1/\delta,
    \]
as required. In the remaining part of this section, we prove Lemma \ref{l:YIJ}.

\begin{proof}[Proof of Lemma~\ref{l:YIJ}]
Let $\mathcal E$ be the set of edges of the torus adjacent to at least one of the vertices from $\mathcal I$.
We will show that $Y_{\mathcal C_{\sT}(0)} \leq |\mathcal E|$.

Let $G=\mathcal C_{\sT}(0)$, and let $\widetilde G= G\setminus \mathcal E$
denote the subgraph of $G$ obtained by removing every edge of $G$ that is in $\mathcal E$.
Note that every vertex from $\mathcal I$ is an isolated vertex in $\widetilde G$.
We claim that the graph $\widetilde G$ does not contain long cycles.
Indeed, assume that there is a long cycle $\pi$ in $\widetilde G$. Since $\widetilde G$ is a subgraph of $G$,
$\pi$ is a long cycle in $G$.
In particular, there exists $z\in\pi$ such that $0$ is connected to $z$ in $G$ by a path that does not use any edges from $\pi$.
Therefore, $z\in\mathcal I$ and $z$ is not an isolated vertex in $\widetilde G$. This is a contradiction.

We have just shown that by removing every edge adjacent to a vertex in $\mathcal I$,
we obtain a subgraph of $\mathcal C_{\sT}(0)$ without long cycles.
This implies that $Y_{\mathcal C_{\sT}(0)} \leq |\mathcal E| \leq 2d |\mathcal I|$.

Further, by the BK inequality, \eqref{eq:expCT} and Theorem \ref{thmNCCExponent},
    \[
    {\mathbb E}_{p_c}|\mathcal I|
    \leq
    {\mathbb E}_{p_c}|\mathcal C_{\sT}(0)|~ {\mathbb P}_{p_c}(0\mbox{ is in a long cycle})
    \leq C_2 V^{-1/3} .\
    \]
\end{proof}

\subsection{Proof of Theorem~\ref{thmExisNCC}(a): existence of long cycles}

We need to show that
there exists $c>0$ such that
    \[
     {\mathbb P}_{p_c}\left(\mbox{there exists a long cycle of length }>c V^{1/3}\right) >c .\
    \]
Take $\varepsilon >0$. The precise value of $\varepsilon$ will be determined later. Define
    \[
     M=|\{x\colon x \mbox{ is in a long cycle of length }>\varepsilon V^{1/3}\}|.
    \]
Then, clearly,
    \begin{equation}
    {\mathbb P}_{\sT, p_c}\left(\exists \mbox{ a long cycle of length }>\varepsilon V^{1/3} \right)
    = {\mathbb P}_{\sT,p_c}(M\neq 0).
    \end{equation}
By the second moment method, we can bound
    \[
    {\mathbb P}_{\sT,p_c}(M\neq 0) \geq \frac{({\mathbb E}_{\sT, p_c}M)^2}{{\mathbb E}_{\sT, p_c}M^2}.
    \]
We first show that ${\mathbb E}_{\sT, p_c}M \geq C_3V^{1/3}$, and then that
${\mathbb E}_{\sT, p_c}M^2 \leq C_4V^{2/3}$.
By translation invariance,
    \[
    {\mathbb E}_{\sT, p_c}M
    =
    V {\mathbb P}_{\sT, p_c}(0 \mbox{ in a long cycle of length }>\varepsilon V^{1/3}) .\
    \]
Recall the definition of $\LC{k}$. 
We write
    \[
    {\mathbb P}_{\sT, p_c}(0 \mbox{ in a long cycle of length }>\varepsilon V^{1/3})
    \geq
    {\mathbb P}_{\sT, p_c}(0 \mbox{ in a long cycle})
    -
    {\mathbb P}_{\sT, p_c}(\LC{\varepsilon V^{1/3}}) .\
    \]
It follows from Theorem~\ref{thmNCCExponent} that
    \[
    {\mathbb P}_{\sT, p_c}(0 \mbox{ in a long cycle}) > C_5V^{-2/3} ,\
    \]
and from Proposition~\ref{prNCCExponent1} that 
    \begin{equation}
    \label{eq:EvarepsilonV13}
    {\mathbb P}_{\sT, p_c}(\LC{\varepsilon V^{1/3}}) < C_6\varepsilon V^{-2/3} .\
    \end{equation}
By taking $\varepsilon$ small enough, we deduce that 
\[
{\mathbb P}_{\sT, p_c}(0 \mbox{ in a long cycle of length }>\varepsilon V^{1/3})
\geq
\frac 12 {\mathbb P}_{\sT, p_c}(0 \mbox{ in a long cycle}) ,\
\]
and the desired lower bound on ${\mathbb E}_{\sT, p_c}M$ follows. 

\medskip

It remains to prove that ${\mathbb E}_{\sT, p_c}M^2 \leq C_4 V^{2/3}$. Since
    \[
    {\mathbb E}_{\sT, p_c}M^2
    =\sum_{x,y} {\mathbb P}_{\sT,p_c}\left(x, y \mbox{ in long cycles of length }>\varepsilon V^{1/3}\right) ,\
    \]
it suffices to show that
\[
 \sum_{x,y} {\mathbb P}_{\sT,p_c}\left(x, y \mbox{ in long cycles}\right) \leq C_4 V^{2/3} .\
\]
We split the above sum, depending on whether the events
$\{x \mbox{ in long cycle}\}$ and
$\{y \mbox{ in long cycle}\}$ occur disjointly or not.
The contribution where these events do occur disjointly can be bounded, using the BK inequality and Theorem~\ref{thmNCCExponent}, by
$C_7V^{2/3}$, so that
    \begin{equation}\label{eq:M2}
    {\mathbb E}_{\sT, p_c}M^2
    \leq C_7V^{2/3}
    +
    \sum_{x,y}{\mathbb P}_{\sT,p_c}\left(x, y \mbox{ in overlapping long cycles}\right),
    \end{equation}
where the event $\{x, y \mbox{ in overlapping long cycles}\}$
indicates that all pairs of long cycles, one of which
contains $x$ and the other $y$, share at least one edge.

\begin{lemma}[Contribution of overlapping cycles]
\label{lem-overlap}
Assume (\ref{prop2pf}).
There exists $C<\infty$ such that
    \begin{equation}
    \label{to-do-lb-exis-NCC}
    \sum_{x,y}{\mathbb P}_{\sT,p_c}\left(x, y \mbox{ in overlapping long cycles}\right)
    \leq CV^{2/3}.
    \end{equation}
\end{lemma}

\begin{proof} 
For a pair of vertices $x,y\in\torus$ and $s\geq 0$, we denote by $\Delta_s(x,y)$ the event that 
$x$ is connected to $y$ by a path in $\torus$ which visits $\partial \Qn{s}(x)$. 

Note that if $x$ and $y$ are in overlapping long cycles, then there exist $u,v$ such that
\begin{itemize}\itemsep1pt
\item[(a)]
$u$ and $v$ both are part of the long cycle that contains $x$ as well as the one that contains $y$, 
\item[(b)]
the connections $x\leftrightarrow u, u\leftrightarrow v, x\leftrightarrow v, y\leftrightarrow u, y\leftrightarrow v$ all occur disjointly,
\item[(c)]
at least one of the events $\Delta_{r/12}(x,u)$, $\Delta_{r/12}(u,v)$, or $\Delta_{r/12}(v,x)$ occur. 
\end{itemize}
Therefore, using symmetry and the BK inequality, we can upper bound the sum in \eqref{to-do-lb-exis-NCC} by
\eqan{
&4\sum_{x,y,u,v}{\mathbb P}_{\sT,p_c}(\Delta_{r/12}(x,u))\tau_{\sT,p_c}(x,v)\tau_{\sT,p_c}(u,v)\tau_{\sT,p_c}(y,u)\tau_{\sT,p_c}(y,v)\nonumber\\
&\qquad+\sum_{x,y,u,v}\tau_{\sT,p_c}(x,u)\tau_{\sT,p_c}(x,v){\mathbb P}_{\sT,p_c}(\Delta_{r/12}(u,v))\tau_{\sT,p_c}(y,u)\tau_{\sT,p_c}(y,v).\label{eq:overlapingcycles_xyuv}
}
By \eqref{eq:torus2pf:n}, we can bound ${\mathbb P}_{\sT,p_c}(\Delta_{r/12}(x,u))
\leq C_8V^{-2/3}$.
Let
    \eqn{
    \nabla_{\sT,p}=\sup_{x,y} \sum_{u,v}\tau_{\sT,p}(x,u)\tau_{\sT,p}(u,v)\tau_{\sT,p}(v,y) .\label{eq:nabla}
    }
It follows from \cite[Theorem 1.6(iii)]{BCHSS05b} and \cite[(1.6)]{HH}
that $\nabla_{\sT,p_c} < C_9$.
Therefore, we can bound the first sum in \eqref{eq:overlapingcycles_xyuv} by
    \eqn{
    C_8V^{-2/3}\sum_{x,y,u,v}\tau_{\sT,p_c}(x,v)\tau_{\sT,p_c}(u,v)\tau_{\sT,p_c}(y,u)\tau_{\sT,p_c}(y,v)
    \leq C_8V^{-2/3}V\nabla_{\sT,p_c}\expec_{\sT,p_c}|{\mathcal C}_{\sT}(0)|
    \stackrel{\eqref{eq:expCT}}\leq C_{10}V^{2/3},
    }
and the second sum in \eqref{eq:overlapingcycles_xyuv} by
    \eqan{
    &C_8V^{-2/3}\sum_{x,y,u,v}\tau_{\sT,p_c}(x,u)\tau_{\sT,p_c}(x,v)\tau_{\sT,p_c}(y,u)\tau_{\sT,p_c}(y,v)\nonumber\\
    &\quad \leq C_8V^{-2/3} V \sup_v\sum_{x,y,u,v}\tau_{\sT,p_c}(0,u)\tau_{\sT,p_c}(u,y)\tau_{\sT,p_c}(y,v)\sum_v\tau_{\sT,p_c}(v)\nonumber\\
    &\quad=C_8V^{-2/3} V  \nabla_{\sT,p_c}\expec_{\sT,p_c}|{\mathcal C}(0)|
    \stackrel{\eqref{eq:expCT}}\leq C_{10}V^{2/3}.\nonumber
    }
These estimates complete the proof.
\end{proof}

\medskip

We continue with the proof of Theorem~\ref{thmExisNCC}(a). 
It follows from \eqref{eq:M2} and Lemma~\ref{lem-overlap} that ${\mathbb E}_{\sT,p_c}M^2 \leq C_4V^{2/3}$.
By the second moment method, we obtain
    \[
    {\mathbb P}_{\sT, p_c}\left(\exists \mbox{ a long cycle of length }>\varepsilon V^{1/3}\right)
    = {\mathbb P}_{\sT,p_c}(M\neq 0)
    \geq \frac{({\mathbb E}_{\sT, p_c}M)^2}{{\mathbb E}_{\sT, p_c}M^2}
    \geq \frac{C_3^2}{C_4}>0 .\
    \]
This completes the proof of the fact that a long cycle of length $>\varepsilon V^{1/3}$ exists with positive probability.
\qed

\medskip

\subsection{Proof of Theorem~\ref{thmExisNCC}(b): non-existence of long cycles}

In this section we prove that, for any positive $\delta$, with positive
probability uniformly in $r$, the clusters of size $>\delta V^{2/3}$ do
not contain any long cycles.
In other words, recalling the definition of $Y_\delta$ in Section~\ref{sec:Ytight}, we will prove the following proposition:
\begin{proposition}[Non-existence of long cycles]
\label{prop:Ydelta=0}
For any positive $\delta$ there exists $c>0$ such that, for all $r\geq 1$,
    \[
     {\mathbb P}_{\sT,p_c}(Y_\delta =0)>c.\
    \]
\end{proposition}

\begin{proof}
For $x\in\torus$, run the following exploration of the edges of $\mathcal C_{\sT}(x)$ started from $x$.
Enumerate the edges of $\torus$. (In the algorithm we describe now, if
there are several edges to choose from, we always pick the edge with the smallest number.)
The first stage of the algorithm is the standard depth-first exploration.
At this stage, after $n$ steps, the algorithm produces
\begin{itemize}
\item
the set of explored vertices $X_n$
(which will be a subset of the vertices of $\mathcal C_{\sT}(x)$),
\item
the set of explored edges $E_n$ (these will be the explored edges, the occupancy of which we will check),
\item
the set of open explored edges $T_n$ and the open cluster induced by these edges, also denoted by $T_n$
(which will be part of the depth-first spanning tree of $\mathcal C_{\sT}(x)$), and
\item
the set of unexplored edges $U_n$ (the algorithm will not check the occupancy of these edges).
\end{itemize}
Further, let $W_n = E_n\cup U_n$.

Take $x\in\torus$. Let $X_0 = T_0 = \{x\}$, $W_0 = \varnothing$.
Let $n\geq 0$. Assume that $X_{n}$, $E_{n}$, $T_{n}$ and $U_n$ are defined.
If there is no edge $\{a,b\}$ with $a\in X_n$ and $\{a,b\}\notin W_n$, then we stop the algorithm and
write $\mathcal A_x = \mathcal A_n$ for all $\mathcal A\in \{X,E,T,U,W\}$.
Otherwise, pick the vertex $a\in X_n$ which is the farthest from $x$ in $T_{n}$ for which there exists $b\in\torus$ such that $\{a,b\} \notin W_n$. Such a vertex, if it exists, is always \emph{unique},
since we explore depth-first.
(We prove this statement in Lemma~\ref{l:dftree:properties}(a) at the end of the section.)

Let $e = \{a,b\}$ be the smallest such edge.
We distinguish two cases:
\begin{enumerate}
 \item
If $b\notin X_n$, then we define $E_{n+1} = E_{n}\cup\{e\}$, $U_{n+1} = U_n$, and {\it check} the occupancy of $e$.
\begin{enumerate}
\item
If $e$ is open, then we define $X_{n+1} = X_n\cup\{b\}$ and $T_{n+1} = T_n\cup\{e\}$.
\item
If $e$ is closed, then we define $X_{n+1} = X_n$ and $T_{n+1} = T_n$.
\end{enumerate}
\item
If $b\in X_n$ (in this case we call $e$ a {\it surplus} edge), then we define
$X_{n+1} = X_n$, $T_{n+1} = T_n$, $E_{n+1} = E_n$, and $U_{n+1}=U_n\cup\{e\}$, and {\it do not check} the occupancy of $e$.
\end{enumerate}
Since the number of edges of $\torus$ is finite, this stage of the algorithm will terminate at some step $N<\infty$.
We then write $\mathcal A_x = \mathcal A_N$ for all $\mathcal A\in \{X,E,T,U,W\}$.
In particular, $X_x$ is the vertex set of $\mathcal C_{\sT}(x)$,
$T_x$ is the ``depth-first'' spanning tree of $\mathcal C_{\sT}(x)$ with root at $x$, and
$W_x$ is the set of edges with at least one end vertex in $\mathcal C_{\sT}(x)$.
The occupancy of edges in $E_x$ is known. In particular, the graph induced by sets
of open edges in $E_x$ is $T_x$. The occupancy of edges in $U_x$ has not been checked.
The sets $E_x$ and $U_x$ are disjoint. Also note that, given the set of unexplored edges
$U_x$, the edges in $U_x$ are open independently of each other.
An example of an edge $\{a,b\}\in U_x$ is given in Figure \ref{fig-1-Ux} below.

\begin{figure}
\begin{center}
\vskip-1cm

\includegraphics[scale=0.70]{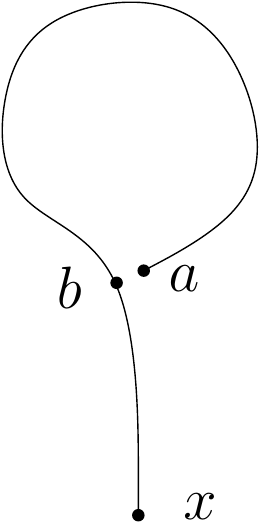}
\caption{Example of an edge $\{a,b\}\in U_x$.}
\label{fig-1-Ux}
\end{center}
\end{figure}

We proceed by describing the second stage of the algorithm.
The aim of this second stage is to select those surplus
edges $\{a,b\}$ that (i) close a long cycle; and (ii) are such that
$x\conn b$ is completely disjoint from the long cycle that is created
by the addition of the edge $\{a,b\}$; and (iii) there are no long cycles
precisely when all these selected edges are closed.

After $n$ steps, the algorithm produces
\begin{itemize}
 \item
the set of open explored edges $G_n$ and the open cluster induced by these edges, also denoted by $G_n$
(which will be a subgraph of $\mathcal C_{\sT}(x)$ without long cycles),
\item
the set of explored edges $F_n$ (which will be a subset of edges
of $U_x$, the occupancy of which we will check), and
\item
the set of special edges $Z_n$ (the algorithm will not check the occupancy
of these edges; each $e\in Z_n$ will have the property that
in the graph $G_n\cup\{e\}$, $e$ is in a long cycle $\pi$,
and there exists a path $\rho$ connecting one of the end-vertices of $e$ to $x$
which is edge disjoint of $\pi$).
\end{itemize}
Note that, according to the first stage of the algorithm,
each edge $e\in U_x$ can be written as $\{a,b\}$ such that the
unique path from $a$ to $x$ in the spanning-tree $T_x$
passes through $b$. (We prove this statement in Lemma~\ref{l:dftree:properties}(b)
at the end of the section.)

Denote by $B_x$ the set of end-vertices with this property, that is,
a vertex $b$ is in $B_x$ if and only if
there exists a vertex $a$ such that the edge $\{a,b\}$ is
in $U_x$ and the unique path from $a$ to $x$ in the spanning tree $T_x$
passes through $b$.

We enumerate the vertices of $B_x$ subject to the following restriction:
a vertex $b\in B_x$ receives a smaller number than $b'\in B_x$ if
the unique path from $b'$ to $x$ in the spanning tree $T_x$ passes through $b$.
This ordering of the vertices in $B_x$ can be better understood by introducing
an auxiliary abstract tree $\mathbf T_x$ rooted at $x$ with the vertex set
$\{x\}\cup B_x$ and the following set of oriented (away from the root) edges:\\
For $b,b'\in B_x$, there is an edge from $b$ to $b'$ in $\mathbf T_x$,
if the unique path from $b'$ to $x$ in the depth-first spanning tree $T_x$
passes through $b$, and the unique path between $b$ and $b'$ in $T_x$
does not contain any other vertices from $B_x$.
With this definition, we can alternatively say that a vertex
$b\in B_x$ has a smaller number than $b'\in B_x$ if
there is an oriented path from $b$ to $b'$ in $\mathbf T_x$.
In other words, we enumerate the vertices of $B_x$ according to
their distance to $x$ in the abstract tree $\mathbf T_x$.
An example of a collection of cycles and the corresponding tree $\mathbf T_x$
is given in Figure \ref{fig-2-Tx} below.

The second stage of the algorithm goes as follows.
Let $G_0 = T_x$, $F_0 = \varnothing$, $Z_0 = \varnothing$ and $B_0 = B_x$.
Assume that, for $n\geq 0$, the sets $G_n$, $F_n$, $Z_n$ and $B_n$ are defined.
If $B_n = \varnothing$, then we stop the algorithm and define
$G_x = G_n$, $F_x = F_n$ and $Z_x = Z_n$.
\begin{figure}
\begin{center}
\vskip-1cm

\includegraphics[scale=0.7]{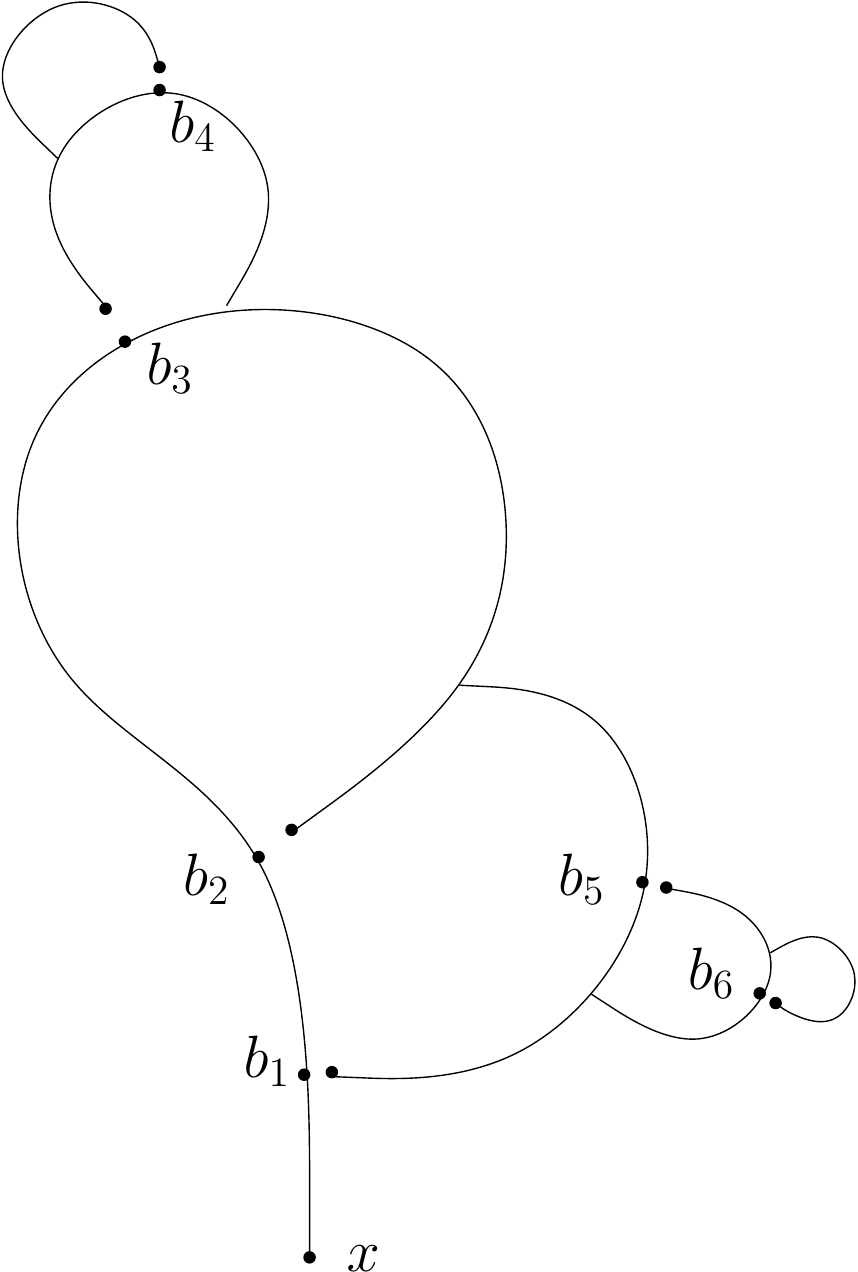}\hskip2cm
\includegraphics[scale=0.7]{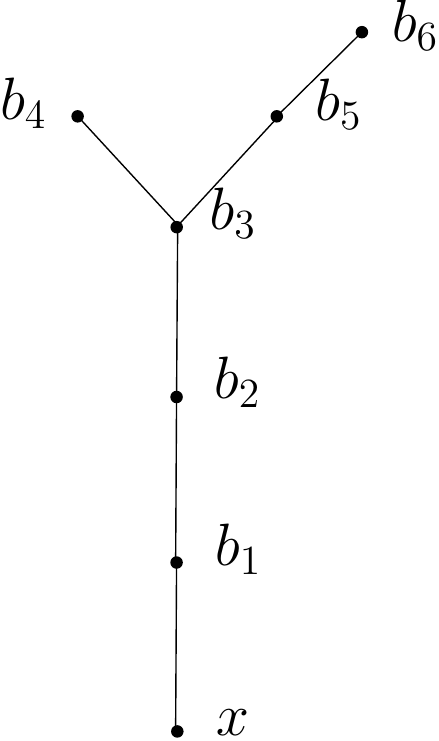}
\caption{Example of a collection of cycles and the corresponding tree $\mathbf T_x$.}
\label{fig-2-Tx}
\end{center}
\end{figure}

Otherwise, pick a vertex $b\in B_n$ with the biggest number.
We distinguish two cases:
\begin{enumerate}
 \item
If there are at least two vertices $a$ and $a'$ such that the edges $\{a,b\}$ and $\{a',b\}$ are in
$U_x\setminus (F_n\cup Z_n)$, then define $B_{n+1} = B_n$ and we select the admissible edge with the smallest number.
(This is the same numbering of the edges of the torus as in the first stage of the algorithm.)
\item
If the vertex $a$ such that the edge $\{a,b\}$ is in $U_x\setminus (F_n\cup Z_n)$ is
unique, then we define $B_{n+1} = B_n\setminus \{b\}$
and select this edge.
\end{enumerate}
Assume that the edge $e = \{a,b\}$ is selected.
\begin{enumerate}
 \item
If the graph $G_n\cup\{e\}$ {\it does not contain} a long cycle, then we define $F_{n+1} = F_n\cup\{e\}$ and $Z_{n+1} = Z_n$ and check the occupancy of $e$.
\begin{enumerate}
\item
If $e$ is open, then define $G_{n+1} = G_n\cup\{e\}$.
\item
If $e$ is closed, then define $G_{n+1} = G_n$.
\end{enumerate}
\item
If the graph $G_n\cup\{e\}$ {\it does contain} a long cycle, then we define $G_{n+1} = G_n$, $F_{n+1} = F_n$ and $Z_{n+1} = Z_n\cup\{e\}$ and {\it do not check} the occupancy of $e$.
Note that, by the special ordering of the vertices in $B_x$, every long cycle of the graph $G_n\cup\{e\}$ passes through $e$ and it is
{\it edge disjoint with the unique path from $b$ to $x$ in the tree $T_x$}.
\end{enumerate}
Since the number of edges of $\torus$ is finite, this stage of the algorithm will terminate at some step $N'<\infty$.
We then write $\mathcal A_x = \mathcal A_{N'}$ for all $\mathcal A\in \{G,F,Z\}$.
The sets $F_x$ and $Z_x$ are disjoint, and their union is $U_x$.
The occupancy of edges in $F_x$ is known. In particular, the graph induced by set of open edges in $E_x\cup F_x$ is $G_x$.
The occupancy of edges in $Z_x$ has not been checked.
In particular, given the set $Z_x$, the edges in $Z_x$ are open independently of each other.
By the definition of $Z_x$, any edge $e\in Z_x$ is in a long cycle in the graph $G_x\cup\{e\}$,
and every long cycle of the graph $G_x\cup\{e\}$ passes through $e$.
Moreover, by the special ordering of the vertices in $B_x$,
any edge $e\in Z_x$ can be written as $\{a,b\}$ so that the unique path
from $b$ to $x$ in the spanning tree $T_x$ is edge disjoint from some long
cycle of the graph $G_x\cup\{e\}$ (but not necessarily from all long cycles
of the graph $G_x\cup\{e\}$).

We run the above defined exploration algorithm for all the open clusters of the torus.
We pick a vertex $x_1$ uniformly on the torus and determine the depth-first spanning tree of the cluster of $x_1$ with root at $x_1$, $T_{x_1}$, the set of
explored edges $E_{x_1}\cup F_{x_1}$, and the set of special edges $Z_{x_1}$.
We then pick a vertex $x_2$ uniformly from the remaining vertices and
determine the depth-first spanning tree of $\mathcal C_{\sT}(x_2)$, $T_{x_2}$,
the set of explored edged $E_{x_2}\cup F_{x_2}$, and the set of special edges $Z_{x_2}$.
We then proceed similarly by selecting $x_3,\ldots, x_M$ and determining $T_{x_i}$, $E_{x_i}\cup F_{x_i}$ and $Z_{x_i}$. Here $M= M(\omega)$ is the number of open clusters in the realization $\omega$.

Given the sets of explored edges $E_{x_i}\cup F_{x_i}$, the number of long cycles
is defined by the status of the special edges $Z_{x_i}$.
In particular, if all the edges in $Z_{x_i}$ are closed, then
$\mathcal C_{\sT}(x_i)$ does \emph{not} contain long cycles.
Note that given the set of explored edges $E_{x_i}\cup F_{x_i}$, the event
that all the edges in $Z_{x_i}$ are closed has probability
    \[
     (1-p_c)^{|Z_{x_i}|} .\
    \]
Also, the size of a cluster is determined by the number of vertices in a spanning tree.
Therefore, (remember that $M$ is the number of open clusters in the torus)
    \begin{equation}\label{eq:YdeltaZxi}
    {\mathbb P}_{p_c}\left(Y_\delta = 0\right)
    =
    {\mathbb E}_{p_c} \left[(1-p_c)^{\sum_{i=1}^M |Z_{x_i}| I(|\mathcal C(x_i)| > \delta V^{2/3})}\right]
    \geq
    (1-p_c)^{{\mathbb E}_{p_c}\left[\sum_{i=1}^M |Z_{x_i}| I(|\mathcal C(x_i)| > \delta V^{2/3})\right]} .\
    \end{equation}
The last step follows from Jensen's inequality.

Let $\mathcal C_{\sss (1)},\ldots, \mathcal C_{\sss(M)}$ be all the clusters of the torus
sorted from the largest (in the number of vertices) to the smallest with ties broken in an arbitrary way.
We will show that
    \begin{equation}\label{eq:fromuniformtosorted}
    {\mathbb E}_{p_c}\left[\sum_{i=1}^M |Z_{x_i}| I(|\mathcal C_{\sT}(x_i)| > \delta V^{2/3})\right]
    =
    {\mathbb E}_{p_c}\left[\sum_{i=1}^M \sum_{x\in\mathcal C_{\sss (i)}}\frac{1}{|\mathcal C_{\sss (i)}|}|Z_{x}| I(|\mathcal C_{\sss(i)}| > \delta V^{2/3})\right] .\
    \end{equation}
Fix a percolation realization on the torus.
Remember the way we select the vertices $x_1,\ldots,x_M$:
select $x_1$ uniformly on the torus, select $x_2$ uniformly on
$\torus\setminus \mathcal C_{\sT}(x_1)$,
select $x_3$ uniformly on $\torus\setminus (\mathcal C_{\sT}(x_1)
\cup\mathcal C_{\sT}(x_2))$, and so on.
Given a percolation realization on the torus, we can
select the vertices $x_1,\ldots,x_M$ in two steps: first
select a permutation $\sigma$ of $\{1,\ldots,M\}$
(the distribution of $\sigma$ is irrelevant to us),
and then select $x_i$ uniformly from $\mathcal C_{\sss (\sigma(i))}$.
Note that the sum $\sum_{i=1}^M |Z_{x_i}| I(|\mathcal C_{\sT}(x_i)| > \delta V^{2/3})$
does not depend on $\sigma$, i.e., on the order in which we select clusters,
and only depends on which points in clusters we select as $x_i$'s.
This implies \eqref{eq:fromuniformtosorted}.

Note that
    \[
    {\mathbb E}_{p_c}\left[\sum_{i=1}^M \sum_{x\in\mathcal C_{\sss(i)}}\frac{1}{|\mathcal C_{\sss(i)}|}|Z_{x}| I(|\mathcal C_{\sss(i)}| > \delta V^{2/3})\right]
    \leq
    \frac{1}{\delta V^{2/3}}
    {\mathbb E}_{p_c}\left[\sum_{i=1}^M \sum_{x\in\mathcal C_{\sss(i)}}|Z_{x}|\right]
    =
    \frac{1}{\delta V^{2/3}} \sum_{x\in\torus} {\mathbb E}_{p_c} [|Z_x|] .\
    \]
Therefore, it follows from \eqref{eq:YdeltaZxi} and \eqref{eq:fromuniformtosorted} that
    \[
    {\mathbb P}_{p_c}\left(Y_\delta = 0\right)
    \geq
    (1-p_c)^{\delta^{-1}V^{1/3} {\mathbb E}_{p_c} [|Z_0|]} .\
    \]
Proposition~\ref{prop:Ydelta=0} follows once we show that
    \begin{equation}\label{eq:ineqforZ0}
    {\mathbb E}_{p_c} [|Z_0|]
    \leq
    C_{11} V^{-1/3} .\
    \end{equation}
Recall that $\mathcal I$ is the set of $z\in \mathcal C_{\sT}(0)$ such that
$\{0\leftrightarrow z\}\circ\{z\mbox { is in a long cycle}\}$, and let
$\mathcal E$ be the set of edges with at least one end-vertex in $\mathcal I$.
By the properties of $Z_0$, if $e\in Z_0$ is open, then $e\in \mathcal E$.
Therefore,
    \[
    {\mathbb E}_{p_c} [|Z_0|]
    =
    \frac{1}{p_c}\sum_e {\mathbb P}_{p_c}(e\in Z_0, e\mbox{ is open})
    \leq
    \frac{1}{p_c}{\mathbb E}_{p_c} [|\mathcal E|]
    \leq
    \frac{1}{p_c}2d{\mathbb E}_{p_c} [|\mathcal I|] .\
    \]
The claim \eqref{eq:ineqforZ0} now follows from Lemma~\ref{l:YIJ}.
This completes the proof of Proposition~\ref{prop:Ydelta=0}.
\end{proof}

In the remainder of this section, we prove some properties of
the exploration algorithm defined in the proof of Proposition~\ref{prop:Ydelta=0}.
Remember the notation used in the description of the algorithm.

\begin{lemma}[Structure depth-first tree]
\label{l:dftree:properties}
(a) For $n\geq 0$, let $\widetilde X_n$ be the set of vertices $a'\in X_n$ for
which there exists $b'\in\torus$ such that the edge $\{a',b'\} \notin W_n$.
For each $n\geq 0$, there exists a unique vertex $a\in \widetilde X_n$
which is the farthest from $x$ in the tree $T_{n}$, and all the other
vertices from $\widetilde X_n$ belong to the unique path from $a$ to the root $x$ in $T_n$.\\
(b) For all $e\in U_x$, there exist $a,b\in X_x$ such that $e = \{a,b\}$
and the unique path from $a$ to $x$ in the tree $T_x$ passes through $b$.
\end{lemma}

\begin{proof}
The proof of part (a) is by induction on $n$.
The result is obvious for $n=0$, since $\widetilde X_0 = \{x\}$.
Assume that the result holds for all $n'\leq n$.
Pick the unique vertex $a\in \widetilde X_n$ which is the farthest from $x$ in $T_{n}$.
Let $\{a,b\}\notin W_n$. (If there are several choices, then
we pick the smallest edge according to the numbering of the edges of the torus.)
If $\{a,b\}$ satisfies 1(a) of the first stage of the algorithm, then
$b\in \widetilde X_{n+1}$, and the unique path from $b$ to $x$ in $T_{n+1}$
contains $\widetilde X_n$, by the induction assumption.
If $\{a,b\}$ satisfies 1(b) or 2 of the first stage of the algorithm,
then either $\widetilde X_{n+1} = \widetilde X_n$ (if there is more
than one edge $\{a,b'\}\notin W_n$) or
$\widetilde X_{n+1} = \widetilde X_n\setminus\{a\}$ (if there is the unique edge
$\{a,b'\} \notin W_n$).
In both cases $\widetilde X_{n+1}$ satisfies the statement
in part (a) of the lemma, by the induction hypothesis.
This completes the proof of (a).\\
To prove part (b), let $e\in U_x$. There exists $n\geq 0$ such that $U_{n+1}\setminus U_n = \{e\}$.
By the definition of the algorithm, there exist $a$ and $b$ such that $e = \{a,b\}$ and
$a$ is the farthest vertex in $\widetilde X_n$ from $x$ in $T_n$.
Note that the edge $e$ satisfies condition 2 of the first stage of the algorithm.
In particular, $b\in X_n$ and $e\notin W_n$. Therefore, $b\in \widetilde X_n$.
The result in part (b) now follows from part (a) of the lemma.
\end{proof}

\section{Proof of Theorem~\ref{thm:ConnInsideBalls}}
\label{sec-pf-thmShortArms}

In this section, we restrict to percolation on $\mathbb Z^d$.
In particular, all the paths are assumed by default to be in $\mathbb Z^d$,
and we write here $\{x\leftrightarrow y\}$ for $\{x\leftrightarrow y~\mbox{in}~\mathbb Z^d\}$
and ${\mathcal C}(x)$ for ${\mathcal C}_{\sZ}(x)$.
This section is organized as follows. In Section \ref{secLemmaShortArms},
we start with some preparatory lemmas based on the techniques in
\cite{KN, KN(IIC)}.
We prove Theorem~\ref{thm:ConnInsideBalls}(a) in Section \ref{sec:ShortConn}, and
Theorem~\ref{thm:ConnInsideBalls}(b) in Section \ref{sec:ShortArms}.

\subsection{Preparatory lemmas}
\label{secLemmaShortArms}
The following lemma produces the factor $\sqrt{\varepsilon}$ that is present
in Theorem~\ref{thm:ConnInsideBalls}(b):

\begin{lemma}\label{lShortArms2}
There exists $C<\infty$ such that for any $\varepsilon > 0$,
positive integer $n$, and $x\in{\mathbb Z}^d$ with $|x| \geq n^2$,
    \[
    {\mathbb P}_{p_c}\big(0\stackrel{\leq \varepsilon n^2}\longleftrightarrow \partial \Qn{n}, 0\leftrightarrow x\big)
    \leq
    C\sqrt\varepsilon {\mathbb P}_{p_c}(0\leftrightarrow x).
    \]
\end{lemma}

\begin{proof}[Proof of Lemma~\ref{lShortArms2}]
This proof is a slight modification of the proof of \cite[Lemma~2.5]{KN(IIC)}.
The event $E = \{0\stackrel{\leq \varepsilon n^2}\longleftrightarrow \partial \Qn{n}\}$ is measurable with respect to
${\mathcal B}_I(\varepsilon n^2) = \{x\in\mathbb Z^d\colon0\stackrel{\leq \varepsilon n^2}\longleftrightarrow x\}$.
Therefore, \cite[Lemma~2.5]{KN(IIC)} implies that for any $x\in{\mathbb Z}^d$ with $|x|$ sufficiently large,
    \[
    {\mathbb P}_{p_c}\left(0\stackrel{\leq \varepsilon n^2}\longleftrightarrow \partial \Qn{n}, 0\leftrightarrow x\right)
    \leq
    C_1\sqrt{\varepsilon n^2{\mathbb P}_{p_c}(E)}{\mathbb P}_{p_c}(0\leftrightarrow x).
    \]
In fact, it follows from the proof of \cite[Lemma~2.5]{KN(IIC)} that the above inequality
holds for all $x\in{\mathbb Z}^d$ with $|x|\geq n^2$.
Finally, remember that ${\mathbb P}_{p_c}(E) \leq C_2n^{-2}$ by (\ref{prop1Arm}).
This completes the proof.
\end{proof}

It follows from Lemma~\ref{lShortArms2} and (\ref{propSizeCn}) that
    \begin{equation}
    \label{eqShortArms1}
    \sum_{x\in \Qn{2n^2}\setminus \Qn{n^2}}
    {\mathbb P}_{p_c}\left(0\stackrel{\leq \varepsilon n^2}\longleftrightarrow \partial \Qn{n}, 0\leftrightarrow x\right)
    \leq
    C_3 \sqrt\varepsilon n^4,
    \end{equation}
which shall be used crucially later on.

We next recall some notation from \cite{KN}.
Recall the definition of a $K$-regular vertex from \cite[Definition~4.1]{KN}:
For $A\subseteq \Z^d$, let ${\mathcal C}(y;A)$ be the set of vertices $z$
such that $y\leftrightarrow z~\mbox{in}~A$.
For $y\in\partial \Qn{n}$ and positive integers $s$ and $K$, 
we say that $y$ is $s$-{\it bad} if ${\mathcal C}(y;\Qn{n})$ satisfies
    \[
    {\mathbb P}_{p_c}\left(|{\mathcal C}(y)\cap \Qn{s}(y)| < s^4\log^7s~|~{\mathcal C}(y;\Qn{n})\right)
    \leq 1-\exp(-\log^2 s).
    \]
We further say that $y\in\partial \Qn{n}$ is $K$-{\it irregular} if 
there exists $s\geq K$ such that $y$ is $s$-bad. Otherwise we say that $y$ is $K$-{\it regular}.

We say that a pair of vertices $(x,y)$ is $(n,K,\varepsilon)$-admissible if the following conditions hold:
(a) $y\in\partial \Qn{n}$ and $x\in \Qn{2n^2}\setminus \Qn{n^2}$;
(b) $0\stackrel{\leq \varepsilon n^2}\longleftrightarrow y$ in $\Qn{n}$ and $y\leftrightarrow x$;
(c) $y$ is $K$-regular; and
(d) the edge $\{y,\widetilde y\}$ is pivotal for the event
$0\leftrightarrow x$, where
$\widetilde y$ is the neighbor of $y$ not in $\Qn{n}$
(if more than one exist, then we choose the first in lexicographical order).

We define $Y(n,K,\varepsilon)$ as the number of $(n,K,\varepsilon)$-admissible pairs.
The random variable $Y(n,K,\varepsilon)$ is very similar to $Y(j,K,L)$ defined in the proof of \cite[Lemma~5.1]{KN}.
\begin{remark}\label{rem:KN:L51}
Note that \cite[Lemma~5.1]{KN} holds for all $M\geq 1$ and not just for $M\geq L^2/2$ as it is stated.
Indeed, \cite[Lemma~5.1]{KN} follows directly from \cite[Lemmas 5.3 and 5.5]{KN},
both of which hold (and are stated) for all $M\geq 1$.
\end{remark}
\begin{lemma}\label{lnKe}
There exists a positive constant $C_4 = C_4(K)$ such that for $K$ sufficiently large,
any positive integer $n$ and $\varepsilon>0$,
    \[
    {\mathbb E}_{p_c}Y(n,K,\varepsilon)
    \geq C_4(K) n^4 {\mathbb E}_{p_c}|\{y\in\partial \Qn{n}\colon0\stackrel{\leq \varepsilon n^2}\longleftrightarrow y~\mbox{in}~\Qn{n}, y~\mbox{is}~K\mbox{-regular}\}|.
    \]
\end{lemma}

\begin{proof}
A word for word repetition of the proof of \cite[Lemma~5.1]{KN} (taking into account
Remark~\ref{rem:KN:L51}) gives Lemma \ref{lnKe}.
Indeed, with the notation of \cite{KN},
the only difference in the proofs arises in the proof (and the statement) of \cite[Lemma~5.3]{KN}, where instead of the event
    \[
    \mathcal E_1 = \{0\leftrightarrow y \mbox{ in }\Qn{j}, y\mbox{ is }K-\mbox{regular and }X^{K-reg}_j = M\} ,\
    \]
we use the event
    \[
    \widetilde{\mathcal E_1} = \{0\stackrel{\leq \varepsilon j^2}\longleftrightarrow y \mbox{ in }\Qn{j}, y\mbox{ is }K-\mbox{regular and }X^{K-reg}_j = M\} .\
    \]
However, the event $\widetilde{\mathcal E_1}$ still can be determined by observing only the edges of $\mathcal C(0;\Qn{j})$.
Therefore, the proof of Lemma~5.3 in \cite{KN} remains unchanged if we replace the event $\mathcal E_1$ with the event $\widetilde{\mathcal E_1}$.
The proof of Lemma~5.5 in \cite{KN} also requires only that the event $\mathcal E_1$ must be determined by observing only the edges of $\mathcal C(0;\Qn{j})$,
and therefore also holds with $\mathcal E_1$ replaced with $\widetilde{\mathcal E_1}$.
\end{proof}

Note that for every $x\in \Qn{2n^2}\setminus \Qn{n^2}$ there exists at most one $y\in\partial \Qn{n}$ such that the pair of vertices $(x,y)$ is $(n,K,\varepsilon)$-admissible.
Therefore,
    \[
    {\mathbb E}_{p_c}Y(n,K,\varepsilon)
    =
    \sum_{x\in \Qn{2n^2}\setminus \Qn{n^2}}
    {\mathbb P}_{p_c}\Big(\exists y\in\partial \Qn{n}\colon(x,y)~\mbox{is}~(n,K,\varepsilon)\mbox{-admissible}\Big).
    \]
If $(x,y)$ is $(n,K,\varepsilon)$-admissible,
then $\{0\stackrel{\leq \varepsilon n^2}\longleftrightarrow y~\mbox{in}~\Qn{n}\}$ and $\{0\leftrightarrow x\}$
both occur.
We use this observation to bound the expected number of $(n,K,\varepsilon)$-admissible pairs from above by
    \[
    \sum_{x\in \Qn{2n^2}\setminus \Qn{n^2}}
    {\mathbb P}_{p_c}\left(\exists y\in\partial \Qn{n}\colon0\stackrel{\leq \varepsilon n^2}\longleftrightarrow y~\mbox{in}~\Qn{n}, 0\leftrightarrow x\right)
    =
    \sum_{x\in \Qn{2n^2}\setminus \Qn{n^2}}
    {\mathbb P}_{p_c}\left(0\stackrel{\leq \varepsilon n^2}\longleftrightarrow \partial \Qn{n}, 0\leftrightarrow x\right).
    \]
We can now combine these bounds with the results of (\ref{eqShortArms1}) and Lemma~\ref{lnKe} to get
    \begin{equation}\label{eqYUpperBound}
    {\mathbb E}_{p_c}\Big|\{y\in\partial \Qn{n}\colon0\stackrel{\leq \varepsilon n^2}\longleftrightarrow y~\mbox{in}~\Qn{n}, y~\mbox{is}~K\mbox{-regular}\}\Big|
    \leq
    \frac{C_3}{C_4(K)}\sqrt\varepsilon
    \end{equation}
for all large enough $K$, positive integers $n$ and for all $\varepsilon>0$ with
the constants $C_3$ from (\ref{eqShortArms1}) (not depending on $K$, $n$ and $\varepsilon$) and $C_4(K)$ from Lemma~\ref{lnKe} (not depending on $n$ and $\varepsilon$). We next investigate the contribution from $K$-irregular $y$'s:

\begin{lemma}
\label{lShortArms3}
For all large enough $K$ and for all $\varepsilon>0$,
    \[
    \limsup_{n\to\infty}\sum_{y\in\partial \Qn{n}}{\mathbb P}_{p_c}\left(0\stackrel{\leq \varepsilon n^2}\longleftrightarrow y~\mbox{in}~\Qn{n}, y\mbox{ is }K\mbox{-irregular}\right)
    \leq \frac{1}{2}\limsup_{n\to\infty}\sum_{y\in\partial \Qn{n}}{\mathbb P}_{p_c}\left(0\stackrel{\leq 2\varepsilon n^2}\longleftrightarrow y~\mbox{in}~\Qn{n}\right).
    \]
\end{lemma}

\begin{proof}[Proof of Lemma~\ref{lShortArms3}]
Recall the definition of an $s$-locally bad vertex from \cite[Definition~4.3]{KN}:
Let $\mathcal T^{loc}_s(y)$ be the event that
(a) for all $z\in \Qn{s}(y)$, $|{\mathcal C}(z;\Qn{s^{2d}}(y))\cap \Qn{s}(y)| < s^4\log^4 s$; and
(b) there exist at most $\log^3 s$ disjoint open paths starting in $\Qn{s}(y)$ and ending at $\partial \Qn{s^{2d}}(y)$.
For $y\in\partial \Qn{n}$ and positive integers $s$ and $K$,
we say that a cluster ${\mathcal C}$ in $\Qn{s^{4d^2}}(y)\cap \Qn{n}$ is a ``spanning cluster'' if 
(a) ${\mathcal C}\cap \Qn{n}$ intersects both $\partial \Qn{s^{4d^2}}(y)$ and $\Qn{s^{2d}}(y)$, or
(b) ${\mathcal C} = {\mathcal C}(y)$.
We say that $y$ is $s$-{\it locally bad} if there exist spanning clusters 
${\mathcal C}_1,\ldots,{\mathcal C}_m$ in $\Qn{s^{4d^2}}(y)\cap \Qn{n}$ such that
    \[
    {\mathbb P}_{p_c}\left(\mathcal T^{loc}_s(y)~|~{\mathcal C}_1,\ldots,{\mathcal C}_m\right) \leq 1 - \exp(-\log^2 s).
    \]
Note that the event that $y$ is $s$-locally bad is determined by the status of the edges in the box $\Qn{s^{4d^2}}(y)\cap \Qn{n}$.
Moreover, it follows from \cite[Claim~4.2]{KN} that if $y$ is not $K$-regular, 
then there exists $s\geq K$ such that $y$ is $s$-locally bad.
Therefore, we need to bound from above the probabilities
\[
{\mathbb P}_{p_c}\left(0\stackrel{\leq \varepsilon n^2}\longleftrightarrow y~\mbox{in}~\Qn{n}, y~\mbox{is}~s\mbox{-locally bad}\right)
\]
for $y\in\partial \Qn{n}$, $s\geq K$ and large enough $n$.

Since we are only interested in large $n$, we may assume that $n\varepsilon > 1$.
We consider two different cases: $2d(s^{4d^2})^d < n$ and $2d(s^{4d^2})^d \geq n$.
We start with the case $2d(s^{4d^2})^d < n$.
Note that in this case the ball $\Qn{s^{4d^2}}$ contains at most $\varepsilon n^2$ edges.
We bound the sum
    \[
    \sum_{y\in\partial \Qn{n}}
    {\mathbb P}_{p_c}\left(0\stackrel{\leq \varepsilon n^2}\longleftrightarrow y~\mbox{in}~\Qn{n}, y~\mbox{is}~s\mbox{-locally bad}\right)
    \]
from above by
    \[
    \sum_{y\in\partial \Qn{n}}
    {\mathbb P}_{p_c}\left(0\stackrel{\leq \varepsilon n^2}\longleftrightarrow \Qn{s^{4d^2}}(y)~\mbox{in}~\Qn{n}, y~\mbox{is}~s\mbox{-locally bad}\right).
    \]
Since the events $\{0\stackrel{\leq \varepsilon n^2}\longleftrightarrow \Qn{s^{4d^2}}(y)~\mbox{in}~\Qn{n}\}$ and
$\{y~\mbox{is}~s\mbox{-locally bad}\}$ depend on the states of edges in disjoint subsets of $\mathbb Z^d$, they are independent.
In particular, the above sum equals
    \[
    \sum_{y\in\partial \Qn{n}}
    {\mathbb P}_{p_c}\left(0\stackrel{\leq \varepsilon n^2}\longleftrightarrow \Qn{s^{4d^2}}(y)~\mbox{in}~\Qn{n}\right)
    {\mathbb P}_{p_c}\left(y~\mbox{is}~s\mbox{-locally bad}\right).
    \]
By \cite[Lemma~1.1]{KN} and the FKG inequality (see, e.g., \cite[Theorem~2.4]{Grimmett}),
there exist a positive constant $C_5$ and a finite constant $C_6$ such that for all $m$ and $z,z'\in\partial \Qn{m}$,
    \[
    {\mathbb P}_{p_c}(z\leftrightarrow z'~\mbox{in}~\Qn{m}) \geq C_5\exp\left(-C_6\log^2 m\right).
    \]
We apply this result to ``extend'' the path
$0\stackrel{\leq \varepsilon n^2}\longleftrightarrow \Qn{s^{4d^2}}(y)~\mbox{in}~\Qn{n}$ to a path
$0\stackrel{\leq 2\varepsilon n^2}\longleftrightarrow y~\mbox{in}~\Qn{n}$:
    \[
    {\mathbb P}_{p_c}\left(0\stackrel{\leq \varepsilon n^2}\longleftrightarrow \Qn{s^{4d^2}}(y)~\mbox{in}~\Qn{n}\right)
    \leq C_7 \exp\left(C_8\log^2 s\right){\mathbb P}_{p_c}\left(0\stackrel{\leq 2\varepsilon n^2}\longleftrightarrow y~\mbox{in}~\Qn{n}\right).
    \]
Here we also use the fact that the number of edges in $\Qn{s^{4d^2}}(y)$ is at most $\varepsilon n^2$,
which implies that if two vertices $z$ and $z'$ in $\Qn{s^{4d^2}}(y)$ are connected by an open path in $\Qn{s^{4d^2}}(y)$
then the length of this path is at most $\varepsilon n^2$.

It follows from \cite[Lemma~4.3]{KN} that
    \[
    {\mathbb P}_{p_c}\left(y~\mbox{is}~s\mbox{-locally bad}\right)
    \leq C_9 \exp\left(-C_{10}\log^4 s\right).
    \]
We now put these bounds together. Let ${\sum}'$ be the sum over all $s$ such that $s\geq K$ and $2d(s^{4d^2})^d < n$. We obtain that
    \[
    {\sum}'
    \sum_{y\in\partial \Qn{n}}
    {\mathbb P}_{p_c}\left(0\stackrel{\leq \varepsilon n^2}\longleftrightarrow y~\mbox{in}~\Qn{n}, y~\mbox{is}~s\mbox{-locally bad}\right)
    \]
is bounded from above by
    \[
    C_{11}\exp\left(-C_{12}\log^4 K\right)
    \sum_{y\in\partial \Qn{n}}
    {\mathbb P}_{p_c}(0\stackrel{\leq 2\varepsilon n^2}\longleftrightarrow y~\mbox{in}~\Qn{n}),
    \]
where the constants $C_{11}$ and $C_{12}$ do not depend on $n$, $\varepsilon$ or $K$.

We now consider the case $2d(s^{4d^2})^d \geq n$. Let $s_0 = (n/2d)^{1/(4d^3)}$.
For $s\geq s_0$, we simply bound
    \[
    \sum_{s=s_0}^\infty\sum_{y\in\partial \Qn{n}}
    {\mathbb P}_{p_c}\left(0\stackrel{\leq \varepsilon n^2}\longleftrightarrow y~\mbox{in}~\Qn{n}, y~\mbox{is}~s\mbox{-locally bad}\right)
    \leq
    |\partial \Qn{n}| \sum_{s=s_0}^\infty \sup_{y\in\partial \Qn{n}}{\mathbb P}_{p_c}\left(y~\mbox{is}~s\mbox{-locally bad}\right).
    \]
We again use \cite[Lemma~4.3]{KN} to bound the above expression by
    \[
    |\partial \Qn{n}| \sum_{s=s_0}^\infty C_9 \exp\left(-C_{10}\log^4 s\right)
    \leq C_{13} \exp\left(-C_{14}\log^4 n\right),
    \]
since $s_0 = (n/2d)^{1/(4d^3)}$, and
where the constants $C_{13}$ and $C_{14}$ do not depend
on $n$, $K$, or $\varepsilon$.

We take $K$ so large that $C_{11}\exp\left(-C_{12}\log^4 K\right) < 1/2$.
Remember \cite[Claim~4.2]{KN} which states that if $y$ is $K$-irregular, then
there exists $s\geq K$ such that $y$ is $s$-locally bad.
Therefore, for such choice of $K$, the sum
    \begin{equation}
    \sum_{y\in\partial \Qn{n}}
    {\mathbb P}_{p_c}\left(0\stackrel{\leq \varepsilon n^2}\longleftrightarrow y~\mbox{in}~\Qn{n}, y~\mbox{is}~K\mbox{-irregular}\right)
   \leq
	\label{eqShortArms3}
    \frac{1}{2}\sum_{y\in\partial \Qn{n}}
    {\mathbb P}_{p_c}\left(0\stackrel{\leq 2\varepsilon n^2}\longleftrightarrow y~\mbox{in}~\Qn{n}\right)
    +
    C_{13} \exp\left(-C_{14}\log^4 n\right).
    \end{equation}
The result of Lemma~\ref{lShortArms3} follows.
\end{proof}

\subsection{Proof of Theorem~\ref{thm:ConnInsideBalls}(a)}
\label{sec:ShortConn}

The proof of Theorem~\ref{thm:ConnInsideBalls}(a) is similar to the proof of Lemma~\ref{lShortArms3}, but easier.
We refer the reader to Section~\ref{secLemmaShortArms} for definitions and notation.
Remember the definition of a $K$-regular vertex from Section~\ref{secLemmaShortArms}.
Let $X_n^{K-reg}$ be the number of $K$-regular vertices on the boundary of $\Qn{n}$
connected to the origin by an open path in $\Qn{n}$.

Let $Y(n,K,L)$ be the random variable defined in the proof of \cite[Theorem~2]{KN}:
We say that a pair of vertices $(x,y)$ are $(n,K,L)$-admissible if the following conditions hold:
(a) $y\in\partial \Qn{n}$ and $x\in \Qn{L}(y)$;
(b) $0\leftrightarrow y$ in $\Qn{n}$ and $y\leftrightarrow x$;
(c) $y$ is $K$-regular; and
(d) the edge $\{ y,\widetilde y\}$ is pivotal for the event
$0\leftrightarrow x$, where $\widetilde y$ is the neighbor of $y$ not in
$\Qn{n}$ (if more than one exist, we choose the first in lexicographical order).
We define $Y(n,K,L)$ as the number of $(n,K,L)$-admissible pairs.

It follows from \cite[Lemma~5.1]{KN} and Remark~\ref{rem:KN:L51} that there exists $C_{15} = C_{15}(K)$ such that
for all large enough $K$ and for all $n$ and $L$,
    \begin{equation}
    \label{EYnKL-bd}
    {\mathbb E}_{p_c}Y(n,K,L) \geq C_{15}(K) L^2 {\mathbb E}_{p_c} X_n^{K-reg}.
    \end{equation}

\begin{lemma}
\label{lem-X-reg-lb}
For all large enough $K$ and for all $n$,
    \[
    {\mathbb E}_{p_c} X_n^{K-reg} \geq \frac{1}{3}\sum_{y\in\partial \Qn{n}}{\mathbb P}_{p_c}(0\leftrightarrow y~\mbox{in}~\Qn{n}).
    \]
\end{lemma}

\begin{proof}
The proof of this lemma is very similar to the proof of Lemma~\ref{lShortArms3}, but simpler.
In the same way as in the proof of Lemma~\ref{lShortArms3} and with the same choice of $K$, we bound
    \[
    \sum_{y\in\partial \Qn{n}}
    {\mathbb P}_{p_c}\left(0\leftrightarrow y~\mbox{in}~\Qn{n}, y~\mbox{is not}~K\mbox{-regular}\right)
    \leq
    \frac{1}{2}\sum_{y\in\partial \Qn{n}}
    {\mathbb P}_{p_c}\left(0\leftrightarrow y~\mbox{in}~\Qn{n}\right)
    +
    C_{13} \exp\left(-C_{14}\log^4 K\right).
    \]
We then use the result of \cite[Lemma~3.1]{KN}: For all positive integers $n$,
    \[
    \sum_{y\in\partial \Qn{n}}{\mathbb P}_{p_c}(0\leftrightarrow y~\mbox{in}~\Qn{n}) \geq 1.
    \]
We increase $K$ if necessary to fulfill the bound $C_{13} \exp\left(-C_{14}\log^4 K\right) < 1/6$.
The result follows.
\end{proof}

Theorem~\ref{thm:ConnInsideBalls}(a) follows in a straightforward way
from \eqref{EYnKL-bd}, Lemma \ref{lem-X-reg-lb} and the fact that
    \[
    {\mathbb E}_{p_c}Y(n,K,n) \leq \sum_{x\in \Qn{2n}} {\mathbb P}_{p_c}(0\leftrightarrow x) \leq C_{16} n^2.
    \]
The last inequality follows from (\ref{propSizeCn}).
\qed

\medskip

\subsection{Proof of Theorem~\ref{thm:ConnInsideBalls}(b)}
\label{sec:ShortArms}

Let
    \[
    F(\varepsilon)
    =
    \limsup_{n\to\infty}\sum_{y\in\partial \Qn{n}}{\mathbb P}_{p_c}\left(0\stackrel{\leq \varepsilon n^2}\longleftrightarrow y~\mbox{in}~\Qn{n}\right).
    \]
Theorem~\ref{thm:ConnInsideBalls}(a) implies that there exists a finite constant $C_{17}$ such that $F(\varepsilon) \leq C_{17}$ for all $\varepsilon>0$.
It follows from (\ref{eqYUpperBound}) and Lemma~\ref{lShortArms3} that for all $\varepsilon>0$,
    \[
    F(\varepsilon) \leq \frac{C_3}{C_4}\sqrt\varepsilon + \frac{1}{2}F(2\varepsilon).
    \]
We apply the above inequality $k$ times to get
    \[
    F(\varepsilon) \leq C_{18}\sqrt\varepsilon + \frac{1}{2^k}F(2^k\varepsilon)
    \leq C_{18}\sqrt\varepsilon + \frac{C_{17}}{2^k},
    \]
with $C_{18} = C_3\sqrt 2/C_4(\sqrt 2 - 1)$ and where we use that $F(2^k\varepsilon)\leq C_{17}$ by
Theorem~\ref{thm:ConnInsideBalls}(a). This inequality holds for any fixed $k$.
We complete the proof of Theorem~\ref{thm:ConnInsideBalls}(b) by taking $k$ such that $2^k\sqrt\varepsilon > 1$.
\qed

\medskip

\section{Proof of Theorem~\ref{thm:ShortConn:r}}\label{sec:ShortConn:r}

\begin{proof}[Proof of \eqref{eq:ShortConn:r}]
Let $k$ and $r$ be positive integers and $z\in\Z^d$. For brevity, we write $r/4$ instead of $\lfloor r/4\rfloor$.
It suffices to prove the result for all large enough $k$.
By Theorem~\ref{thm:ConnInsideBalls}(b), there exist $A$ and $K$ such that for any $k\geq K$, 
	\begin{equation}\label{eq:AK}
	\sum_{y\in\partial \Qn{A\sqrt{k}}}
	{\mathbb P}_{p_c}\left(0\stackrel{\leq k}\longleftrightarrow y~\mbox{in}~\Qn{A\sqrt{k}}\right) 
	\leq \frac 12 .\
	\end{equation}
Indeed, fix $\varepsilon>0$ such that the right hand side of \eqref{eq:ConnInsideBalls(b)} is strictly smaller than $1/2$. 
Then by \eqref{eq:ConnInsideBalls(b)}, inequality \eqref{eq:AK} holds for all large $k$ with $A = 1/\sqrt{\varepsilon}$. 

From now on we assume that $k\geq K$.
We first consider the case $A\sqrt{k} \leq r/8$. 
We write 
	\begin{eqnarray*}
	\sum_{x\in\Z^d,~x\stackrel{r}\sim z,~|x|\geq r/4}
	{\mathbb P}_{\sZ, p_c}\left(0\stackrel{\leq k}\longleftrightarrow x\mbox{ in } \Z^d\right)
	&\leq
	&\sum_{n=0}^\infty \sum_{x\stackrel{r}\sim z,~r/4 + rn\leq |x|<r/4 + r(n+1)}
	{\mathbb P}_{\sZ, p_c}\left(0\stackrel{\leq k}\longleftrightarrow x\mbox{ in } \Z^d\right) \\
	&\leq
	&\sum_{n=0}^\infty C_1 (n+1)^{d-1} \sup_{r/4 + rn\leq |x|<r/4 + r(n+1)}
	{\mathbb P}_{\sZ, p_c}\left(0\stackrel{\leq k}\longleftrightarrow x\mbox{ in } \Z^d\right) .
	\end{eqnarray*}
Here we use the fact that, uniformly in $r$, the number of $x\stackrel{r}\sim z$ such that
$r/4 + rn\leq |x|<r/4 + r(n+1)$ is of order $(n+1)^{d-1}$.

Take $n\geq 0$ and $x\in\Z^d$ with $r/4 + rn\leq |x|<r/4 + r(n+1)$.
Let $M_n = \lfloor \frac{r/8 + rn}{A\sqrt{k}} \rfloor \geq 1$,
since $A\sqrt{k} \leq r/8$.  

Note that the event $\{0\stackrel{\leq k}\longleftrightarrow x\mbox{ in }\Z^d\}$ implies the existence of
$x_1,\ldots, x_{M_n}$ such that for all $i\in\{1,\ldots,M_n\}$, 
$x_i\in \partial \Qn{A\sqrt{k}}(x_{i-1})$ (where we assume $x_0 = 0$)
and the following connections all occur disjointly:
	\[
	\{x_{i-1}\stackrel{\leq k}\longleftrightarrow x_i\mbox{ in }\Qn{A\sqrt{k}}(x_{i-1})\}, 
	\quad i\in\{1,	\ldots,M_n\},
	\quad
	\mbox{and}
	\quad
	\{x_{M_n}\leftrightarrow x\} .\
	\]
By the BK inequality, translation invariance, \eqref{eq:AK}, the fact that $|x_{M_n} - x| \geq r/8$, and \eqref{prop2pf}, for any $n\geq 0$, 
	\[
	\sup_{r/4 + rn\leq |x|<r/4 + r(n+1)}
	{\mathbb P}_{\sZ, p_c}\left(0\stackrel{\leq k}\longleftrightarrow x\mbox{ in } \Z^d\right)
	\leq C_2 r^{2-d} \left(\frac 12\right)^{M_n} .\
	\]
Putting all the bounds together and using that $M_n\geq n+M_0$ (since $A\sqrt{k} \leq r/8$), 
we have in the case 
$A\sqrt{k} \leq r/8$ that 
	\[
	\sum_{x\in\Z^d,~x\stackrel{r}\sim z,~|x|\geq r/4} 
	{\mathbb P}_{\sZ, p_c}\left(0\stackrel{\leq k}\longleftrightarrow x\mbox{ in } \Z^d\right)
	\leq
	C_1C_2 r^{2-d}\sum_{n=0}^\infty (n+1)^{d-1}\left(\frac 12\right)^{M_n}
	\leq 
	C_3 r^{2-d} 2^{-\frac{r}{8A\sqrt{k}}} .\
	\]
This finishes the proof of \eqref{eq:ShortConn:r} in the case $A\sqrt{k} \leq r/8$. 
\medskip

It remains to consider the case $A\sqrt{k} > r/8$.
We write 
	\begin{eqnarray*}
	\sum_{x\in\Z^d,~x\stackrel{r}\sim z,~|x|\geq r/4} 
	{\mathbb P}_{\sZ, p_c}\left(0\stackrel{\leq k}\longleftrightarrow x\mbox{ in } \Z^d\right)
	&\leq
	&\sum_{x\in\Qn{8A\sqrt{k}},~x\stackrel{r}\sim z,~|x|\geq r/8} 
	{\mathbb P}_{\sZ, p_c}\left(0\leftrightarrow x\mbox{ in } \Z^d\right)\\
	&+
	&\sum_{n=8}^\infty \sum_{x\stackrel{r}\sim z,~nA\sqrt{k}< |x|\leq (n+1)A\sqrt{k}} 
	{\mathbb P}_{\sZ, p_c}\left(0\stackrel{\leq k}\longleftrightarrow x\mbox{ in } \Z^d\right) .\
	\end{eqnarray*}
Using \eqref{propSizeCnr}, the first sum can be bounded from above by $C_4 k/r^d$. 

Take $n\geq 8$ and $x\in\Z^d$ with $nA\sqrt{k}< |x|\leq (n+1)A\sqrt{k}$.
Note that the event $\{0\stackrel{\leq k}\longleftrightarrow x\mbox{ in }\Z^d\}$ implies the existence of
$x_1,\ldots, x_{n-1}$ such that for all $i\in\{1,\ldots,n-1\}$, 
$x_i\in \partial \Qn{A\sqrt{k}}(x_{i-1})$ (where we assume $x_0 = 0$)
and the following connections all occur disjointly:
	\[
	\{x_{i-1}\stackrel{\leq k}\longleftrightarrow x_i\mbox{ in }\Qn{A\sqrt{k}}(x_{i-1})\}, 
	\qquad i\in\{1,	\ldots,n-1\},
	\qquad
	\mbox{and}
	\qquad
	\{x_{n-1}\leftrightarrow x\} .\
	\]
By the BK inequality,
	\begin{multline*}
	\sum_{x\stackrel{r}\sim z,~nA\sqrt{k}< |x|\leq (n+1)A\sqrt{k}} 
	{\mathbb P}_{\sZ, p_c}\left(0\stackrel{\leq k}\longleftrightarrow x\mbox{ in } \Z^d\right) \\
	\leq
	\sum_{x_1,\ldots,x_{n-1}}\prod_{i=1}^{n-1}
	{\mathbb P}_{\sZ, p_c}\left(x_{i-1}\stackrel{\leq k}\longleftrightarrow x_i\mbox{ in }
	\Qn{A\sqrt{k}}(x_{i-1})\right)
	\sum_{x\stackrel{r}\sim z,~nA\sqrt{k}< |x|\leq (n+1)A\sqrt{k}}
	{\mathbb P}_{\sZ, p_c}\left(x_{n-1}\leftrightarrow x\mbox{ in } \Z^d\right) .\
	\end{multline*}
Note that for any choice of $x_1,\ldots, x_{n-1}$ as above, we have $r/8< A\sqrt{k} \leq |x_{n-1} - x| \leq 2nA\sqrt{k}$. 
Therefore, by translation invariance and \eqref{propSizeCnr}, we obtain 
	\begin{eqnarray*}
	\sum_{x\stackrel{r}\sim z,~nA\sqrt{k}< |x|\leq (n+1)A\sqrt{k}}
	{\mathbb P}_{\sZ, p_c}\left(x_{n-1}\leftrightarrow x\mbox{ in } \Z^d\right)
	&\leq
	&\sum_{x\in\Qn{2nA\sqrt{k}},~x\stackrel{r}\sim z-x_{n-1},~|x|\geq r/8}
	{\mathbb P}_{\sZ, p_c}\left(0\leftrightarrow x\mbox{ in } \Z^d\right)\\
	&\leq 
	&\frac{C_5k}{r^d} .\
	\end{eqnarray*}
Putting all the bounds together and using translation invariance and \eqref{eq:AK}, we get
	\[
	\sum_{x\in\Z^d,~x\stackrel{r}\sim z,~|x|\geq r/4} 
	{\mathbb P}_{\sZ, p_c}\left(0\stackrel{\leq k}\longleftrightarrow x\mbox{ in } \Z^d\right)
	\leq
	\frac{C_4k}{r^d} + 
	\sum_{n=8}^\infty \frac{C_5k}{r^d} \left(\frac 12\right)^{n-1}
	\leq 
	\frac{C_6k}{r^d} .\
	\]
This finishes the proof of \eqref{eq:ShortConn:r} in the case $A\sqrt{k}>r/8$. 
\end{proof}

\medskip

\begin{proof}[Proof of \eqref{eq:ShortConn:3r}]
We bound the sum over $u$, $v$, and $w$ by distinguishing three cases: $|u|\geq r/4$, $|v-w|\geq r/4$, and the remaining term. 

If $|v-w|\geq r/4$, then \eqref{eq:ShortConn:3r} follows by applying \eqref{eq:ShortConn:r}. 
Indeed, by \eqref{eq:ShortConn:r} we get one factor of $k/r^d$ from summing over $w$, and
by \eqref{KN-ball} two factors of $k$ from the remaining two sums.

If $|u|\geq r/4$, we let $u'=u+w$ and $v'=v+w$, and replace the sums over $u$ and $v$ by sums 
over $u'$ and $v'$ and use translation invariance 
to obtain that \eqref{eq:ShortConn:3r} over this range equals
	\begin{equation}\label{eq:ShortConn:3r-rep}
	\sum_{w\stackrel{r}\sim z,~|w|\geq 3r/4}\sum_{u',v'\colon |u'-w|\geq r/4}
	{\mathbb P}_{\sZ, p_c}\left(w\stackrel{\leq k}\longleftrightarrow u'\mbox{ in } \Z^d\right)
	{\mathbb P}_{\sZ, p_c}\left(u'\stackrel{\leq k}\longleftrightarrow v'\mbox{ in } \Z^d\right)
	{\mathbb P}_{\sZ, p_c}\left(0\stackrel{\leq k}\longleftrightarrow v'\mbox{ in } \Z^d\right)
	\leq \frac{Ck^3}{r^d} ,\
	\end{equation}
where we use that $|u'-w|=|u|\geq r/4$, together with \eqref{eq:ShortConn:r} and 
\eqref{KN-ball}.

It remains to bound the sum over all $u$, $v$, and $w$ with $w\stackrel{r}\sim z$, $|w|\geq 3r/4$, 
$|u|<r/4$, and $|v-w|<r/4$, which we denote by ${\sum}'$.
By the triangle inequality, we have $|u-v|\geq r/4$. 
We write 
\begin{multline*}
{\sum}'
{\mathbb P}_{\sZ, p_c}\left(0\stackrel{\leq k}\longleftrightarrow u\mbox{ in } \Z^d\right)
{\mathbb P}_{\sZ, p_c}\left(u\stackrel{\leq k}\longleftrightarrow v\mbox{ in } \Z^d\right)
{\mathbb P}_{\sZ, p_c}\left(v\stackrel{\leq k}\longleftrightarrow w\mbox{ in } \Z^d\right)\\
=
{\sum}'
{\mathbb P}_{\sZ, p_c}\left(0\stackrel{\leq k}\longleftrightarrow u\mbox{ in } \Z^d\right)
{\mathbb P}_{\sZ, p_c}\left(u\stackrel{\leq k}\longleftrightarrow v\mbox{ in } \Z^d\right)
{\mathbb P}_{\sZ, p_c}\left(0\stackrel{\leq k}\longleftrightarrow v-w\mbox{ in } \Z^d\right) .\
\end{multline*}
Note that $v\stackrel{r}\sim v-w+z$ and $|u-v|\geq r/4$. 
With the change of variables $(v,x) = (v,v-w+z)$, we observe that the above sum is bounded from above by 
\[
\sum_{u,x}\sum_{v\stackrel{r}\sim x,~|v-u|\geq r/4}
{\mathbb P}_{\sZ, p_c}\left(0\stackrel{\leq k}\longleftrightarrow u\mbox{ in } \Z^d\right)
{\mathbb P}_{\sZ, p_c}\left(u\stackrel{\leq k}\longleftrightarrow v\mbox{ in } \Z^d\right)
{\mathbb P}_{\sZ, p_c}\left(0\stackrel{\leq k}\longleftrightarrow x\mbox{ in } \Z^d\right) .\
\]
Applying \eqref{eq:ShortConn:r} to the sum over $v$, and then \eqref{KN-ball} to the sums over $u$ and $x$, we obtain that 
the above sum is bounded from above by $(C_7k/r^d)\cdot C_7k\cdot C_7k$. 
Putting all the cases together, we arrive at \eqref{eq:ShortConn:3r}. 
\end{proof}

\section{Proof of Theorem \ref{thmTwoPointTorus}}\label{sec:thmTwoPointTorus}

\begin{proof}[Proof of \eqref{eq:torus2pf}]
Take $x\in\mathbb T_r^d$. Let $k = \lfloor V^{1/3}\rfloor$. We write
\[
{\mathbb P}_{p_c}\left(0\leftrightarrow x~\mbox{in}~\mathbb T_r^d\right)
=
{\mathbb P}_{p_c}\left(0\stackrel{\leq 3k}\longleftrightarrow x~\mbox{in}~\mathbb T_r^d\right) 
+ 
{\mathbb P}_{p_c}\left(\{0\leftrightarrow x~\mbox{in}~\mathbb T_r^d\}\setminus 
\{0\stackrel{\leq 3k}\longleftrightarrow x~\mbox{in}~\mathbb T_r^d\}\right) .\
\]
It follows from Proposition~\ref{prCoupling}, \eqref{eq:ShortConn:r}, and the choice of $k$ that
\[
{\mathbb P}_{p_c}\left(0\stackrel{\leq 3k}\longleftrightarrow x~\mbox{in}~\mathbb T_r^d\right) 
\leq 
\tau_{\sZ,p_c}(0,x) + C_1 V^{-2/3} .\
\]

Note that the event 
$\{0\leftrightarrow x~\mbox{in}~\mathbb T_r^d\}\setminus \{0\stackrel{\leq 3k}\longleftrightarrow x~\mbox{in}~\mathbb T_r^d\}$ 
implies that 
(a) $\partial B_{\sT,k}(0)\neq \varnothing$, (b) $\partial B_{\sT,k}(x)\neq \varnothing$, and
(c) $B_{\sT,k}(0)$ and $B_{\sT,k}(x)$ do not intersect.
Let $G'$ be the subgraph of $\torus$ obtained by removing all edges needed to calculate $B_{\sT,k}(x)$. Note that the events (a)-(c) imply that $\partial B_{\sT,k}(x)\neq \varnothing$ and 
$\partial B_{\sT,k}^{\sss G'}(0)\neq \varnothing$.
By \eqref{KN-intrinsic},
	\[
	\sup_{G} {\mathbb P}_{p_c}\left(\partial B_{\sT,k}^{\sss G}(0)\neq \varnothing\right)
	\cdot {\mathbb P}_{p_c}\left(\partial B_{\sT,k}(x)\neq \varnothing\right)
	\leq
	\left(C_2/k\right)^2  \leq C_3 V^{-2/3} .\    
	\]
This completes the proof of \eqref{eq:torus2pf}.
\end{proof}

\medskip

\begin{proof}[Proof of \eqref{eq:torus2pf:n}]
Let $x\in\torus$. Let $n$ be a positive integer smaller than $r/2$.
We distinguish two cases: $|x| < 2n/3$ and $|x| \geq 2n/3$.
In the first case, we observe that the event $\{0\leftrightarrow x\mbox{ by a path which visits }\partial \Qn{n}\}$ implies that
there exists a vertex $y\in\partial \Qn{n}$ such that
\[
\{0\leftrightarrow y\mbox{ in }\Qn{n}\}
\circ
\{y\leftrightarrow x\} .\
\]
By the BK inequality, Theorem~\ref{thm:ConnInsideBalls}(a) and \eqref{eq:torus2pf},
    \eqan{
    &{\mathbb P}_{p_c}(0\leftrightarrow x~\mbox{in}~\mathbb T_r^d\mbox{ by a path which visits }\partial \Qn{n})\\
    &\qquad \leq
    \sum_{y\in \partial \Qn{n}} {\mathbb P}_{p_c}(0\leftrightarrow y\mbox{ in }\Qn{n})\cdot
    {\mathbb P}_{p_c}(y\leftrightarrow x\mbox{ in }\mathbb T_r^d)
    \leq
    C_4(\sup_{y\in\partial \Qn{n}}\tau_{\sZ,p_c}(y,x) + C_5 V^{-2/3}) .\nonumber
    }
Since $|x| < 2n/3$, the distance between $x$ and any $y\in\partial\Qn{n}$ is at least $n/3$. Therefore,
by \eqref{prop2pf}, we have $\tau_{\sZ,p_c}(y,x) \leq C_6 n^{2-d}$, 
and \eqref{eq:torus2pf:n} follows.

\medskip

In the case $|x| \geq 2n/3$, we simply use the bound 
\[
{\mathbb P}_{p_c}(0\leftrightarrow x~\mbox{in}~\mathbb T_r^d\mbox{ by a path which visits }\partial \Qn{n})
\leq
{\mathbb P}_{p_c}(0\leftrightarrow x~\mbox{in}~\mathbb T_r^d)
\stackrel{\eqref{eq:torus2pf}}\leq 
\tau_{\sZ,p_c}(0,x) + C_5V^{-2/3} .\
\]
Since $|x| \geq 2n/3$, we obtain by \eqref{prop2pf} that 
$\tau_{\sZ,p_c}(0,x) \leq C_7 n^{2-d}$, 
and \eqref{eq:torus2pf:n} follows.
\end{proof}

\paragraph{Acknowledgements.}
We thank Asaf Nachmias for his comments on parts of the manuscript.
The work of RvdH and AS was supported in part by the Netherlands
Organisation for Scientific Research (NWO).
The work of AS was further supported by a grant of the `Excellence Fund'
of Eindhoven University of Technology, as well as by the grant ERC-2009-AdG 245728-
RWPERCRI.

\end{document}